\documentclass[reqno]{amsart}%
\usepackage{amsfonts}
\usepackage{verbatim}
\usepackage{xcolor}
\usepackage{amsmath}
\usepackage{amssymb}
\usepackage{graphicx}
\usepackage{accents}%
\providecommand{\U}[1]{\protect\rule{.1in}{.1in}}
\newtheorem{theorem}{Theorem}[section]

\newtheorem{corollary}[theorem]{Corollary}
\newtheorem{definition}[theorem]{Definition}
\newtheorem{proposition}[theorem]{Proposition}
\newtheorem{remark}[theorem]{Remark}
\newtheorem{lemma}[theorem]{Lemma}
\newtheorem{example}[theorem]{Example}
\newtheorem{hypothesis}[theorem]{Hypothesis}
\numberwithin{equation}{section}
\begin{document}
\title[KdV equation]{Soliton theory and Hankel operators }
\author{Sergei Grudsky}
\address{Departamento de Matematicas, CINVESTAV del I.P.N. Aportado Postal 14-740,
07000 Mexico, D.F., Mexico.}
\email{grudsky@math.cinvestav.mx.}
\author{Alexei Rybkin}
\address{Department of Mathematics and Statistics, University of Alaska Fairbanks, PO
Box 756660, Fairbanks, AK 99775}
\email{arybkin@alaska.edu}
\thanks{The first author is supported by PROMEP (M\'{e}xico) via "Proyecto de Redes",
by CONACYT grant 180049, and by Federal program \textquotedblleft Scientific
and Scientific-Pedagogical Personnel of Innovative Russia for the years
2007-2013\textquotedblright\ (contract No 14.A18.21.0873) }
\thanks{The second author is supported in part by the NSF under grant DMS 1009673.}
\date{October, 2013}

\begin{abstract}
Soliton theory and the theory of Hankel (and Toeplitz) operators have stayed
essentially hermetic to each other. This paper is concerned with linking
together these two very active and extremely large theories. On the
prototypical example of the Cauchy problem for the Korteweg-de Vries (KdV)
equation we demonstrate the power of the language of Hankel operators in which
symbols are conveniently represented in terms of the scattering data for the
Schrodinger operator associated with the initial data for the KdV equation.
This approach yields short-cuts to already known results as well as to a
variety of new ones (e.g. wellposedness beyond standard assumptions on the
initial data) which are achieved by employing some subtle results for Hankel
operators. \newline

\end{abstract}
\maketitle

\section{Preface}

The title reflecting the final results of the work should have been formulated
as something like \textquotedblleft On the Inverse Scattering Transform for
the Korteweg-de Vries equation on the line with essentially arbitrary initial
data decaying sufficiently rapidly on the right half line\textquotedblleft.
However, that would have narrowed the number of potential readers down to a
rather small group of specialists interested in the so-called step-like
initial data, which was not our intention. Hoping to draw attention of a much
larger community of mathematicians and theoretical physicists whose lexicon
intersects with ours, we have risked compilation of the title merely from the
names of two enormously large areas which we have the goal to connect.

The name soliton theory in the title of the paper underlines that our main
results belong to this theory. The term Hankel operators, in turn, should
suggest that the exposition is based upon the theory of Hankel operators.
While having experienced a boom at the same time, these two theories have not
shown much of interaction. It is our main goal to demonstrate that the
language of Hankel/Toeplitz operators is very natural for soliton theory
(completely integrable systems) with lots of potentials. We would also like to
capture the attention of the Hankel/Toeplitz operator community who may find
something new in the theoretical aspect, but above all get acquainted with
applications of Hankel operators to the inverse scattering transform for
integrable systems, which may have a stimulating influence on the development
of Hankel and Toeplitz operators.

Our purpose determines the style of the paper: maximally self-contained
exposition with recall of basic definitions and auxiliary results.

We concentrate solely on the Korteweg-de Vries (KdV) case but it should be
quite clear to anyone familiar with the area that our approach by no means is
restricted to this case. Moreover, we believe that the interplay between
soliton theory and Hankel operators may be even more interesting and fruitful
for some other integrable systems with richer than KdV structures.

\section{Introduction}

Soliton theory originated in the mid 60s from the fundamental
Gardner-Greene-Kruskal-Miura discovery of what we now call the inverse
scattering transform (IST) for the KdV equation. It is regarded as a major
achievement of the 20s century connecting different branches of pure
mathematics and theoretical physics with numerous applications ranging from
hydrodynamics and nonlinear optics to astrophysics and elementary particle
theory (see, e.g. the classical books \cite{AC91}, \cite{NovikovetalBook}).
Conceptually, the IST is similar to the Fourier transform. In the context of
the Cauchy problem for the KdV equation on the full line%

\begin{equation}
\partial_{t}q-6q\partial_{x}q+\partial_{x}^{3}q=0, \label{KdV}%
\end{equation}%
\begin{equation}
q\left(  x,0\right)  =q\left(  x\right)  , \label{KdVID}%
\end{equation}
the IST method consists, as the standard Fourier transform method, of\ three steps:

\textbf{Step 1. }(direct transform)%
\[
q\left(  x\right)  \longrightarrow S_{q},
\]
where $S_{q}$ is a new set of variables which turns (\ref{KdV}) into a simple
first order linear ODE for $S_{q}(t)$ with the initial condition
$S_{q}(0)=S_{q}$.

\textbf{Step 2. }(time evolution)
\[
S_{q}\longrightarrow S_{q}\left(  t\right)  .
\]

\textbf{Step 3. }(inverse transform)\textbf{\ }%
\[
S_{q}\left(  t\right)  \longrightarrow q(x,t).
\]

Similar methods have also been developed for many other evolution nonlinear
PDEs, which are referred to as completely integrable.\footnote{There is no
precise meaning of complete integrability but the question \textquotedblleft
What is integrability?" has drawn much attention (see e.g. the survey
\cite{Its2003}).} \ Each of steps 1-3 involves solving a linear equation that
allows us to analyze integrable systems at the level unreachable by neither
direct numerical methods nor standard PDE techniques. The study of a large
variety of realizations of these steps for different integrable systems and
initial conditions (including the analysis of the information that the IST
yields about these systems) constitutes the core of soliton theory.

In the classical IST for (\ref{KdV})-(\ref{KdVID}), when $q$ is rapidly
decaying as $\left\vert x\right\vert \rightarrow\infty$ (the so-called short
range),$\ S_{q}$ is a set of scattering data associated with the
Schr\"{o}dinger operator $\mathbb{L}_{q}=-\partial_{x}^{2}+q$. By solving the
Schr\"{o}dinger equation $\mathbb{L}_{q}u=k^{2}u$ one finds $S_{q}=\left\{
R(k),(\kappa_{n},c_{n})\right\}  $, where $R\left(  k\right)  ,\ k\in
\mathbb{R}$, is the reflection coefficient and $(\kappa_{n},c_{n}%
),\ n=1,2,..,N$, are so-called bound state data associated with with
eigenvalues $-\kappa_{n}^{2}$ of \thinspace$\mathbb{L}_{q}$. Step 2 readily
yields
\begin{equation}
S_{q}(t)=\left\{  R(k)\exp\left(  8ik^{3}t\right)  ,\;\kappa_{n},c_{n}%
\exp\left(  8\kappa_{n}^{3}t\right)  \right\}  . \label{time evol}%
\end{equation}
Step 3 amounts to solving the inverse scattering problem of recovering the
potential $q\left(  x,t\right)  $ (which now depends on $t\geq0$) from
$S_{q}(t)$. This procedure comes with an explicit formula, called determinant
or Dyson,%

\begin{equation}
q\left(  x,t\right)  =-2\partial_{x}^{2}\log\det\left(  I+\mathbb{H}\left(
x,t\right)  \right)  , \label{eq1.3}%
\end{equation}
where $\mathbb{H}\left(  x,t\right)  $ is the Hankel operator $\mathbb{H}%
\left(  \varphi_{x,t}\right)  $ with symbol
\begin{equation}
\varphi_{x,t}(k)=R(k)\xi_{x,t}(k)+\sum_{n=1}^{N}\frac{c_{n}\xi_{x,t}%
(i\kappa_{n})}{\kappa_{n}+ik}. \label{symb}%
\end{equation}
Here $\xi_{x,t}(k)=\exp\{i(8k^{3}t+2kx)\}$ solely carries the dependence on
$\left(  x,t\right)  $. This puts us in the context of the theory of Hankel
operators and Steps 1-3 can now be combined to read%
\begin{equation}
q(x)\longrightarrow\mathbb{H}(\varphi_{x,t})\longrightarrow q(x,t).
\label{steps 1-3}%
\end{equation}

Note that there are many other methods to carry out Step 3. Historically, the
first one, called Gelfand-Levitan-Marchenko, amounts to working with
$\mathbb{H}\left(  x,t\right)  $ in the form of an integral (Marchenko)
operator which kernel is the Fourier transform of $\varphi_{x,t}$. The most
contemporary one is based upon the Riemann-Hilbert problem which is solved
using techniques of singular integral equations. Overall, while both look
quite similar, the latter is arguably much more powerful as the Fourier
transform in a sense smears the dependence on $\left(  x,t\right)  $. Our
approach also starts out from a Riemann-Hilbert problem (the basic scattering
relation (\ref{basic scatt identity})) which we solve in terms of Hankel
operators in the form (\ref{eq4.1}), but not in the Marchenko form. This gives
us a direct access to the well-developed theory of Hankel operators where, in
fact, the integral form is not used much either.

Thus (\ref{steps 1-3}) suggests the relevance of soliton theory to Hankel
operators which theory is also extremely large. While both theories have been
developed at about same time, there has been very little interaction between
the two. The connection we have discussed so far is rather skin deep and does
not offer immediate benefits to either theory. This is likely true if we stay
within the realm of short range initial data $q$ but our goal is to go far
beyond this.

Let us put our goal into historic context. An analog of the IST for periodic
$q$'s was found in \cite{Novikov74}. It rallies on the Floquet theory for
$\mathbb{L}_{q}$ and analysis of Riemann surfaces and hence is much more
complex than the short range case\footnote{Unification of these two cases is
offered in \cite{AC91} as an open problem.}. However, as emphasized in
\cite{KrichNovikov99}, (\ref{KdV})-(\ref{KdVID}) is completely integrable
essentially only in these two cases. In fact, the question whether any
well-posed problem (initial value, boundary value, etc.) for (\ref{KdV}%
)\footnote{Or any other integrable system.} could be solved by a suitable IST,
has been raised in one form or another by many (see e.g. \cite{AC91},
\cite{KrichNovikov99}, \cite{McLeodOlver83}) and some regard it as a major
unsolved problem. Once we are outside of the scattering or periodic situations
many real complications arise. The main question is of course what $S_{q}$
should be. A large amount of effort has been put into developing the IST on
intervals (see, e.g., the recent sequel \cite{Fokas2012},
\cite{LenellsFokas2012}, \cite{LenellsFokasII 2012}). We only mention that
$S_{q}$ consists of the certain spectral functions (some of which depend on
the initial data and some on the boundary data) related through a nonlinear
algebraric equation (the global relation\footnote{Solving this relation is the
only nonlinear step in the IST procedure.}).

Here we are interested in the initial value problem on the whole line but with
$q$ outside of classes of short range or periodic functions. Much of known
(rigorous) results are on some sort of \textquotedblleft hybrids" of these two
cases. Namely, physically important cases of short-range perturbations of a
step function\footnote{A bore wave type initial profile.} (see e.g.
\cite{Cohen1984}, \cite{EgorovaetalNon2013}, \cite{Hruslov76}, \cite{KK94},
\cite{Venak86} ) and two crystals fused together\footnote{Soliton propogation
on a periodic backgound.} (\cite{Teschl1}, \cite{Teschl2}). Steps 1-3 are much
more complicated than in the short range case and require subtle analysis
which has been completed only recently. Certain cases of slowly decaying
profiles have also received considerable attention (see e.g.
\cite{GesztesyDuke92}, \cite{Marchenko91}, \cite{Matveev2002}), but this
situation is much less understood.

In the present paper we deal with initial data subject to

\begin{hypothesis}
\label{hyp1.1}Let $q$ be a real locally integrable function subject to

\begin{enumerate}
\item[(1)] (boundedness from below)
\begin{equation}
\inf\operatorname{Spec}\left(  \mathbb{L}_{q}\right)  =-h_{0}^{2}>-\infty;
\label{Cond1}%
\end{equation}

\item[(2)] (decay at $+\infty$) For some positive weight function $w\left(
x\right)  \geq x$%
\begin{equation}
\int^{\infty}w\left(  x\right)  \left\vert q\left(  x\right)  \right\vert
dx<\infty. \label{Cond2}%
\end{equation}

\end{enumerate}
\end{hypothesis}

We show that under Hypothesis \ref{hyp1.1}\footnote{In fact, we will be a bit
conservative regarding $w$ for some inessential reasons.} (\ref{KdV}%
)-(\ref{KdVID}) is globally well-posed and completely integrable in the sense
that (\ref{steps 1-3}) can be explicitly realized. Note that our class of
initial profiles, which we call step-like, is extremely large. The condition
\begin{equation}
\operatorname*{Sup}\limits_{x}\int_{x}^{x+1}\max\left(  -q,0\right)  <\infty,
\label{Cond_q}%
\end{equation}
is sufficient for (\ref{Cond1}) and is also necessary if $q\leq0$. Therefore,
any $q$ subject to Hypothesis \ref{hyp1.1} is essentially bounded from below,
decay sufficiently fast at $+\infty$ but arbitrary otherwise.

The main feature of our situation is that we can do one sided\ scattering
theory and define a suitable (right) reflection coefficient $R\left(
k\right)  $. The problem is that $R$ need not have smoothness and decay
properties that the machinery of the classical IST relies on. We overcome
these issues by extracting a part $A$ from $R$ (Propositions
\ref{step like refl} and \ref{props of R}) which mimics the irregular behavior
of $R\left(  k\right)  $ for real $k$ but admits an analytic continuation into
the upper half plane. The rest of $R\left(  k\right)  $ is (in a sense) small
and easily controlled by the behavior of $q$ at $+\infty$. What actually makes
this split work is that $A$ can be chosen to keep all the necessary
information about the negative spectrum of the whole operator $\mathbb{L}%
_{q}.$\footnote{I.e., $(\kappa_{n},c_{n})$ in the classical case.} This split
(written in a different form) was a crucial ingredient of our
\cite{RybCommPDEs2013} in dealing with similar to Hypothesis \ref{hyp1.1}
conditions but under extra hard-to-verify assumptions. The problem is that
Hypothesis \ref{hyp1.1} does not rule out the case $\left\vert R\left(
k\right)  \right\vert =1$ for a.e. real $k$~which further complicates the
situation. In the quantum mechanical sense, such $q$'s are completely
nontransparent for plane waves coming from $+\infty$. Examples include (1)
functions growing (arbitrarily fast) at $-\infty$ (not quite physical), (2)
Gaussian white noise on a left half line (like the stock market), (3) certain
sparse sequences of bumps (Pearson blocks \cite{Pearson78}), and (4) certain
(random) slowly decaying at $x\rightarrow-\infty$ functions (Kotani potentials
\cite{KoU88}), to mention just four. Note that if $R\left(  k\right)  $ is
unimodular then all previously known approaches to step-like initial
conditions break down in a serious way. Our approach developed in
\cite{RybNON2011} could handle this case but some inconvenient conditions had
to be imposed to rule out the possibility for $\mathbb{H}(\varphi_{x,t})$ to
have eigenvalue $-1$ for some $\left(  x,t\right)  $. Otherwise, as it follows
from (\ref{eq1.3}), $q\left(  x,t\right)  $ will develop a double pole type
singularity at that point causing (\ref{KdV})-(\ref{KdVID}) to be ill-posed in
many respects. One of our main contributions is showing that under Hypothesis
\ref{hyp1.1} $-1$ never belongs to the spectrum of $\mathbb{H}(\varphi_{x,t})$
and we prove it by means of Hankel/Toeplitz operators. It is curious to note
that the analyticity of $\xi_{x,t}(k)=\exp\{i(8k^{3}t+2kx)\}$ is not essential
here but rather its membership in the Sarason algebra $H^{\infty}+C$ (see
Section \ref{H^infty+C} for definition), a fundamental object of the theory of
Hankel/Toeplitz operators (see, e.g. \cite{BotSil06},\cite{Nik2002}%
,\cite{Peller2003}). The latter is established in \cite{DybGru2002} in a
totally different context. It is this nontrivial fact that opens an access to
the powerful machinery of Hankel/Toeplitz operators.

We also crucially use the Adamyan-Arov-Krein theory (see, e.g. \cite{Nik2002},
\cite{Peller2003}) which provides us with a beautiful way to compute singular
numbers of Hankel operators that, in turn, gives very accurate error estimates
for KdV solutions. The analyticity of $\xi_{x,t}(k)$, not its membership in
the Sarason algebra, becomes important here.

It is the use of the Sarason algebra and Adamyan-Arov-Krein theory that has
let us conclude our study \cite{RybCommPDEs2013}, \cite{RybNON2011} of
complete integrability of (\ref{KdV})-(\ref{KdVID}) with step-like initial
data. Moreover we see that many other results and problems of soliton theory
become more transparent once translated into the language of Hankel operators.
We are now convinced that the theory of Hankel/Toeplitz operators has
potentially much more to offer to soliton theory encouraging closer
interaction between two theories. This point of view is also reinforced by the
recent papers \cite{Gerard2012}, \cite{Gerard2010} where a totally different
set of results on soliton theory was obtained using Hankel operators.

Our goal was to make the paper as self-contained and rigorous as possible
supplemented with proofs of known results whenever it is instructive and can
be done in a nice way. This always presents a challenge if you also want to
keep the volume reasonable. We are not sure if our goal is achieved but we
tried our best.

The paper is organized as follows. In Section \ref{notaiton} we merely list
some of our notation and conventions and give a short review of Hardy and
Gevrey classes. Sections \ref{hankel} and \ref{H^infty+C} are devoted to
reviewing some background information and preparing some facts about Hankel
operators needed in the following sections. In Section \ref{sec6} we recall
some basics of the classical IST and reformulate it in terms of Hankel
operators. In Section \ref{Refl} we study the analytic structure of the
reflection coefficient which is crucially used in the following sections. In
Section \ref{Our hankel} we state and prove some principal properties of our
specific Hankel operator. In Section \ref{s-numbers} we use the
Adamyan-Arov-Krein classical theory to obtain subtle relations between the
decay of singular numbers of our Hankel operator and properties of the initial
data $q$. Section \ref{main section} is devoted to our main result, Theorem
\ref{MainThm}. Its corollaries, as well as related discussions and historical
comments are given in Section \ref{corollaries}. In the final Section
\ref{last sect} we state some open problems.

\section{Notation and Function Classes\label{notaiton}}

\subsection{Basic notation and conventions\label{basic notation}}

We follow standard notation accepted in Analysis. For number sets:
$\mathbb{N}_{0}=\left\{  0,1,2,...\right\}  $, $\mathbb{R}$ is the real line,
$\mathbb{R}_{\pm}=(0,\pm\infty)$, $\mathbb{C}$ is the complex plane,
$\mathbb{C}^{\pm}=\left\{  z\in\mathbb{C}:\pm\operatorname{Im}z>0\right\}  $.
$\overline{z}$ is the complex conjugate of $z.$

Besides number sets, black board bold letters will also be used for (linear)
operators. In particular, $\mathbb{I}$ denotes the identity operator.
$\mathbb{A}^{\ast}$ stands for the adjoint of a linear operator $\mathbb{A}$
on a Hilbert space. For a given compact operator $\mathbb{A}$ on a Hilbert
space, we recall that its $n$-th singular value $s_{n}\left(  \mathbb{A}%
\right)  $ is defined as the $n$-th eigenvalue of the operator $\left(
\mathbb{A}^{\ast}\mathbb{A}\right)  ^{1/2}$. We say that $\mathbb{A}$ is in
the Shatten-von Neumann class $\mathfrak{S}_{p},0<p\leq\infty$, if $\left\{
s_{n}\left(  \mathbb{A}\right)  \right\}  \in l^{p}$. (Here $l^{p}$ stands for
the space of all sequences $\left(  x_{n}\right)  $ such that $\Sigma
_{n}\left\vert x_{n}\right\vert ^{p}<\infty$). We write $\mathbb{A}\geq0$ if
$\left\langle \mathbb{A}f,f\right\rangle \geq0$ for any $f$ from the domain of
$\mathbb{A}$. ($\left\langle \mathbb{\cdot},\cdot\right\rangle $ stands for
the inner product). $\mathbb{A}>0$ means that $\ \mathbb{A}\geq0$ and
$\left\langle \mathbb{A}f,f\right\rangle =0$ iff $f=0$. We write
$\mathbb{A}\geq\mathbb{B}$ if $\mathbb{A}-\mathbb{B}\geq0$.

Some other notation: $\chi\left(  x\right)  $ is the Heaviside function,
$\chi_{c}\left(  x\right)  :=\chi\left(  x-c\right)  $, and
\begin{equation}
\xi_{x,t}(k):=\exp\{i(8k^{3}t+2kx)\},\ \ \xi_{x}(k):=\xi_{x,0}(k)=\exp
\{2ikx\}, \label{cubic exp}%
\end{equation}
is a fundamental to the IST function.

We frequently (but not always) arrange the variables of a function in the
order of their importance. E.g. $f\left(  x,p\right)  $ should suggest that
$x$ is the main variable and $p$ is a parameter. To reduce the amount of
clutter we will often drop variables of functions whenever it causes no
confusion, abbreviate $\int f\left(  x\right)  dx=\int f$ and, when
appropriate, write $y\lesssim_{a}x$ in place of $y\leq C_{a}x$ with some
$C_{a}>0$ dependent on $a$ but independent of $x$. If $C$ is a universal
constant we then write $y\lesssim x$.

We use $\rightrightarrows$ to denote uniform convergence. In particular, we
agree to write $f_{n}\left(  z\right)  \rightrightarrows f\left(  z\right)  $
in $\mathbb{C}^{+}$ if $f_{n}\left(  z\right)  $ converges to $f\left(
z\right)  $ uniformly on compact subsets of $\mathbb{C}^{+}$ containing no
singularities of $f_{n}\left(  z\right)  ,f\left(  z\right)  $.

We have the following agreement on $\pm$ statements: $P_{\pm}\Rightarrow
Q_{\pm}$ means two separate statements $P_{+}\Rightarrow Q_{+}$,
$P_{-}\Rightarrow Q_{-}$. We then use $P_{\pm}$ as a single noun. If it is
used as a plural noun then $P_{\pm}$ means $P_{+}$ and $P_{-}$. Since
quantities labeled with $+$ will appear more frequently whenever convenient we
also drop $+$ and just write $P~$for $P_{+}$. (But we will never suppress the
subscript $-$.)

\subsection{Basic function classes}

As usual, $L^{p}\left(  S\right)  ,\ 0<p\leq\infty$, is the Lebesgue space on
a set $S$. Typically $S=\mathbb{R}$ which justifies the abbreviation
\[
\int\overset{\operatorname*{def}}{=}\int_{\mathbb{R}},\ \ \ L^{p}%
\overset{\operatorname*{def}}{=}L^{p}\left(  \mathbb{R}\right)  .
\]
And ($\partial_{x}^{n}:=\partial^{n}/\partial x^{n},\ n\in\mathbb{N}_{0}$)
\begin{align*}
&  C\overset{\operatorname*{def}}{=}\left\{  f:f\text{ is continuous on
}\mathbb{R}\text{, }\lim_{x\rightarrow\infty}f\left(  x\right)  =\lim
_{x\rightarrow-\infty}f\left(  x\right)  \neq\pm\infty\right\}  ,\ \\
&  C^{n}\overset{\operatorname*{def}}{=}\left\{  f:\partial_{x}^{n}f\in
C\right\}  ,\ n\in\ \mathbb{N}_{0};\ C^{\infty}\overset{\operatorname*{def}%
}{=}\cap_{n\in\ \mathbb{N}_{0}}C^{n}\text{.}%
\end{align*}

\subsection{Hardy classes}

To translate our problem into the language of Hankel/Toeplitz operators some
common definitions and facts are in order \cite{Garnett}.

A function $f$ analytic in $\mathbb{C}^{\pm}$ is in the Hardy space $H_{\pm
}^{p}$ for some $0<p\leq\infty$ if
\[
\Vert f\Vert_{H_{\pm}^{p}}^{p}\overset{\operatorname*{def}}{=}\sup_{y>0}\Vert
f(\cdot\pm iy)\Vert_{p}<\infty.
\]
We remind that by our convention we set $H^{p}=H_{+}^{p}.$

It is a fundamental fact of the theory of Hardy spaces that any $f\left(
z\right)  \in H_{\pm}^{p}$ with $0<p\leq\infty$ has non-tangential boundary
values $f\left(  x\pm i0\right)  $ for almost every (a.e.) $x\in\mathbb{R}$
and
\begin{equation}
\Vert f\Vert_{H_{\pm}^{p}}=\Vert f\left(  \cdot\pm i0\right)  \Vert_{L^{p}%
}\overset{\operatorname*{def}}{=}\left\Vert f\right\Vert _{p}.
\label{H^p norm}%
\end{equation}

Classes $H_{\pm}^{\infty}$ and $H_{\pm}^{2}$ will be particularly important.
$H_{\pm}^{\infty}$ is the algebra of uniformly bounded in $\mathbb{C}^{\pm}$
functions and $H_{\pm}^{2}$ is the Hilbert space with the inner product
induced from $L^{2}$:
\[
\langle f,g\rangle_{H_{\pm}^{2}}=\langle f,g\rangle_{L^{2}}=\left\langle
f,g\right\rangle =\int f\bar{g}.
\]

It is well-known that $L^{2}=H^{2}\oplus H_{-}^{2},$ the orthogonal (Riesz)
projection $\mathbb{P}_{\pm}$ onto $H_{\pm}^{2}$ being given by%
\begin{equation}
(\mathbb{P}_{\pm}f)(x)=\pm\frac{1}{2\pi i}\lim_{\varepsilon\rightarrow0+}%
\int\frac{f(s)ds}{s-(x\pm i\varepsilon)}\overset{\operatorname*{def}}{=}%
\pm\frac{1}{2\pi i}\int\frac{f(s)ds}{s-(x\pm i0)}. \label{eq3.1}%
\end{equation}
Of course
\begin{equation}
\mathbb{P}_{\pm}^{\ast}=\mathbb{P}_{\pm},\;\mathbb{P}_{\pm}^{2}=\mathbb{P}%
_{\pm},\;\mathbb{P}_{+}+\mathbb{P}_{-}=\mathbb{I}. \label{eq3.2}%
\end{equation}
Notice that for any $f\in H^{2}$ and $\lambda\in\mathbb{C}^{+}$
\begin{equation}
\mathbb{P}_{-}\frac{f\left(  \cdot\right)  }{\cdot-\lambda}=\mathbb{P}%
_{-}\frac{f(\cdot)-f(\lambda)}{\cdot-\lambda}+\mathbb{P}_{-}\frac{f(\lambda
)}{\cdot-\lambda}=\frac{f(\lambda)}{\cdot-\lambda}. \label{P_}%
\end{equation}
The operators given by (\ref{eq3.1}) remain bounded from $L^{p}$ to $H_{\pm
}^{p}$ for $1<p<\infty$. For $L^{\infty}$ the (regularized) Riesz projection
\begin{align}
(\widetilde{\mathbb{P}}_{\pm}f)(x)  &  =(x+i)\left(  \mathbb{P}_{\pm}\frac
{1}{\cdot+i}f\right)  (x),\label{eq3.3}\\
&  =\pm\frac{1}{2\pi i}\int\left(  \frac{1}{s-(x\pm i0)}-\frac{1}{s+i}\right)
f(s)ds,\;\ \ f\in L^{\infty}, \label{eq3.4}%
\end{align}
is clearly well-defined and $\widetilde{\mathbb{P}}_{\pm}f$ is analytic in
$\mathbb{C}^{\pm}$. Moreover $\widetilde{\mathbb{P}}_{\pm}$ is bounded from
$L^{\infty}$ to $\text{BMOA}(\mathbb{C}^{\pm})$. I.e.
\begin{equation}
f\in L^{\infty}\implies\widetilde{\mathbb{P}}_{\pm}f\in\text{BMOA}%
(\mathbb{C}^{\pm}). \label{BMOA}%
\end{equation}
Here $\text{BMOA}(\mathbb{C}^{\pm})$ stands for the well-known class of
analytic in $\mathbb{C}^{\pm}$ functions from $\text{BMO}(\mathbb{R})$. The
space $\text{BMO}(\mathbb{R})$ (Bounded Mean Oscillation) consists of locally
integrable functions $f$ on $\mathbb{R}$ satisfying ($I$ is a bounded
interval)
\[
\Vert f\Vert_{\text{BMO}}=\sup_{I\in\mathbb{R}}\frac{1}{|I|}\int_{I}%
|f-f_{I}|<\infty,\ \ \ f_{I}\overset{\operatorname*{def}}{=}\frac{1}{|I|}%
\int_{I}f.
\]
It is important to us that
\begin{equation}
\widetilde{\mathbb{P}}_{+}f+\widetilde{\mathbb{P}}_{-}f=f,\;f\in L^{\infty},
\label{eq3.5}%
\end{equation}
which immediately follows from (\ref{eq3.2}), (\ref{eq3.3}), and
(\ref{eq3.4}). One can also see from (\ref{eq3.4}) that if $f\in L^{2}$ then
$\widetilde{\mathbb{P}}_{+}f=\mathbb{P}_{+}f+\operatorname*{const},\;$and
$\widetilde{\mathbb{P}}_{+}f=\mathbb{P}_{+}f=f$ if $f\in H^{2}$.

We will occasionally use Blaschke products
\begin{equation}
B(z)=\prod\limits_{n\geq1}b_{n}(z),\;b_{n}(z)=\frac{1+z_{n}^{2}}{|1+z_{n}%
^{2}|}\frac{z-z_{n}}{z-\overline{z_{n}}}. \label{eq3.8}%
\end{equation}
Recall that the product\footnote{Assuming that $i$ is not a zero of $B$. If
$i$ is a zero then (\ref{eq3.8}) should be modified accordingly.} in
(\ref{eq3.8}) is uniformly convergent on compact sets in $\mathbb{C}^{+}$ iff
the Blaschke condition%
\begin{equation}
\sum_{n\geq1}\frac{\operatorname{Im}z_{n}}{1+|z_{n}|^{2}}<\infty
\label{Blaschke cond}%
\end{equation}
is satisfied. Moreover $B\in H^{\infty}$, $\left\vert B(z)\right\vert \leq1$
if $z\in\mathbb{C}^{+}$, $B$ is unimodular a.e. on $\mathbb{R}$, and has zeros
in $\mathbb{C}^{+}$ at $\left\{  z_{n}\right\}  .$

Most of our Blaschke products will have purely imaginary zeros $\left\{
ix_{n}\right\}  $ of multiplicity $1$. In this case (\ref{Blaschke cond}) is
equivalent to $\left\{  x_{n}\right\}  \in l^{1}$. Observe that such $B$ is
continuous on the real line away from $0$. The following simple statement is a
product analog of the Weierstrass M-test.

\begin{proposition}
\label{Product analog of m-test} Let $\left\{  x_{n}\left(  p\right)
\right\}  $ be a positive sequence dependent on a parameter $p$ (say
positive). Suppose that $x_{n}\left(  p\right)  \leq\lim_{p\rightarrow\infty
}x_{n}\left(  p\right)  =:x_{n}<\infty$ for each $n$. If $\left\{
x_{n}\right\}  \in l^{1}$ then
\[
B_{p}\left(  z\right)  :=\prod\limits_{n\geq1}\frac{z-ix_{n}\left(  p\right)
}{z+ix_{n}\left(  p\right)  }\underset{p\rightarrow\infty}{\rightrightarrows
}B\left(  z\right)  =\prod\limits_{n\geq1}\frac{z-ix_{n}}{z+ix_{n}}\text{ in
}\mathbb{C}^{+}.
\]

\end{proposition}

\begin{proof}
Setting $b_{n}^{p}\left(  z\right)  :=\frac{z-ix_{n}\left(  p\right)
}{z+ix_{n}\left(  p\right)  }$,$\ b_{n}\left(  z\right)  :=\frac{z-ix_{n}%
}{z+ix_{n}}$ we have for each $z$ away from $\left\{  ix_{n}\right\}  $
\[
\left\vert B_{p}-B\right\vert =\left\vert B\right\vert \ \left\vert
\prod\limits_{n\geq1}b_{n}^{p}/b_{n}-1\right\vert \leq\left\vert
\prod\limits_{n\geq1}b_{n}^{p}/b_{n}-1\right\vert .
\]
We are done if we show that on each compact set $K$ in $\mathbb{C}%
^{+}\diagdown\cup_{p>0}\left\{  ix_{n}\left(  p\right)  \right\}  $
\begin{align*}
f_{p}\left(  z\right)   &  :=\log\prod\limits_{n\geq1}b_{n}^{p}\left(
z\right)  /b_{n}\left(  x\right) \\
&  =\sum_{n\geq1}\log\left(  1+\frac{i\left(  x_{n}-x_{n}\left(  p\right)
\right)  }{z-ix_{n}}\right)  +\sum_{n\geq1}\log\left(  1+\frac{i\left(
x_{n}-x_{n}\left(  p\right)  \right)  }{z+ix_{n}\left(  p\right)  }\right) \\
&  \rightrightarrows0,\ \ \ p\rightarrow\infty.
\end{align*}
But uniformly on $K$
\begin{align*}
\left\vert f_{p}\left(  z\right)  \right\vert  &  \lesssim\sum_{n\geq1}\left(
\left\vert z-ix_{n}\right\vert ^{-1}+\left\vert z+ix_{n}\left(  p\right)
\right\vert ^{-1}\right)  \left(  x_{n}-x_{n}\left(  p\right)  \right) \\
&  \lesssim_{K}\sum_{n\geq1}\left(  x_{n}-x_{n}\left(  p\right)  \right)
\end{align*}
and the standard Weierstrass M-test applies.
\end{proof}

Finally, a set $S\subset\mathbb{C}^{+}$ is called a uniqueness set for $H^{2}$
if for any $f\in H^{2}$
\[
\left.  f\right\vert _{S}=0\Longrightarrow f=0.
\]
Otherwise $S$ is called a nonuniqueness set for $H^{2}$. Obviously, a
nonuniqueness set must satisfy the Blaschke condition. Thus, if a set $S$ is
uncountable then it is a uniqueness set and if $S$ is countable then it must
fail the Blaschke condition.

\subsection{Gevrey classes}

We will need the Gevrey classes $G^{\alpha},\alpha\geq1,$ of smooth functions
$f$ such that%

\[
\left\vert \partial_{x}^{n}f\left(  x\right)  \right\vert \lesssim_{f}%
Q_{f}^{n}\left(  n!\right)  ^{\alpha}\text{ for all }x\text{ and }n,
\]
with some $Q_{f}>0.$ Note that $G^{1}$ is the set of real analytic function.
Following \cite{Dyn76} we call $F\left(  x,y\right)  $ a pseudoanalytic
extension of $f\left(  x\right)  $ to $\mathbb{C}$ if
\[
F\left(  x,0\right)  =f\left(  x\right)  \text{ and }\overline{\partial
}F\left(  x,y\right)  \rightarrow0,y\rightarrow0,
\]
where $\overline{\partial}\overset{\operatorname*{def}}{=}\left(  1/2\right)
\left(  \partial_{x}+i\partial_{y}\right)  $ \ The statement $f\in G^{\alpha}$
is equivalent \cite{Dyn76} to the statement that $f$ admits a pseudoanalytic
extension $F$ such that for some $Q>0$%
\begin{equation}
\left\vert \overline{\partial}F\left(  x,y\right)  \right\vert \lesssim
_{f}\exp\left\{  -Q\left\vert y\right\vert ^{-\frac{1}{\alpha-1}}\right\}  .
\label{lambda_bar}%
\end{equation}

\section{Hankel Operators, basic definitions\label{hankel}}

A Hankel operator is an infinitely dimensional analog of a Hankel matrix, a
matrix whose $(j,k)$ entry depends only on $j+k$. I.e. a matrix $\Gamma$ of
the form
\[
\Gamma=\left(
\begin{array}
[c]{cccc}%
\gamma_{1} & \gamma_{2} & \gamma_{3} & ...\\
\gamma_{2} & \gamma_{3} & ... & \\
\gamma_{3} & ... &  & \\
... &  &  & \gamma_{n}%
\end{array}
\right)  .
\]
Definitions (and properties) of Hankel operators depend on specific spaces and
need not be equivalent. We consider Hankel operators on $H^{2}$ (c.f.
\cite{Nik2002}, \cite{Peller2003}).

Let%
\[
(\mathbb{J}f)(x)\overset{\operatorname*{def}}{=}f(-x)
\]
be the operator of reflection on $L^{2}$. It is clearly an isometry with the
obvious properties%
\begin{align}
\mathbb{J}^{\ast}  &  =\mathbb{J},\;\mathbb{J}^{2}=\mathbb{I},\;\mathbb{J}%
^{-1}=\mathbb{J}.\label{eq4.7}\\
\mathbb{J}\left(  \varphi f\right)   &  =(\mathbb{J}\varphi)\mathbb{J}%
f,\;\varphi\in L^{\infty},\;f\in L^{2}\label{eq4.8}\\
\mathbb{JP}_{\mp}  &  =\mathbb{P}_{\pm}\mathbb{J}. \label{eq4.9}%
\end{align}

\begin{definition}
[Hankel and Toeplitz operators]\label{def4.1}Let $\varphi\in L^{\infty}$. The
operators $\mathbb{H}(\varphi)$ and $\mathbb{T}(\varphi)$ defined by
\begin{equation}
\mathbb{H}(\varphi)f=\mathbb{JP}_{-}\varphi f,\;\text{and }\mathbb{T}%
(\varphi)f=\mathbb{P}_{+}\varphi f,\ \ f\in H^{2}, \label{eq4.1}%
\end{equation}
are called respectively the Hankel and Toeplitz operators with the symbol
$\varphi$.
\end{definition}

\bigskip Due to (\ref{eq4.9}), both $\mathbb{H}(\varphi)$ and $\mathbb{T}%
(\varphi)$ act from $H^{2}$ to $H^{2}$. Note that while $\mathbb{H}(\varphi)$
and $\mathbb{T}(\varphi)$ look alike, they are different parts of the
multiplication operator
\begin{equation}
\varphi f=\mathbb{JH}(\varphi)f+\mathbb{T}(\varphi)f,\ \;f\in H^{2},
\label{eq4.2}%
\end{equation}
and therefore are quite different. The Toeplitz operator will play only an
auxiliary role in our consideration.

Directly from the definition, $\Vert\mathbb{H}(\varphi)\Vert\leq\Vert
\varphi\Vert_{\infty}$ but much more subtle statements will be required.

\begin{theorem}
[Widom, 1960]\label{th4.4}Let $\varphi$ be unimodular. Then $\Vert
\mathbb{H}(\varphi)\Vert<1$ iff $\mathbb{T}(\varphi)$ is left invertible.
\end{theorem}

The sufficiency in this theorem is a direct consequence of (\ref{eq4.2}) but
the necessity is a really deep result (see \cite{BotSil06}, sect. 2.20).

We will occasionally have to deal with certain unbounded symbols {(more
exactly, from BMO}) which nevertheless produce bounded Hankel operators. In
such cases we define $\mathbb{H}(\varphi)$ first on the set
\begin{equation}
\mathfrak{H}_{2}\overset{\operatorname*{def}}{=}\left\{  f\in H^{2}:f\in
C^{\infty},\ f\left(  z\right)  =o\left(  z^{-2}\right)  ,\ z\rightarrow
\infty,\ \operatorname{Im}z\geq0\right\}  ,\ \label{dense set in H}%
\end{equation}
dense \cite{Garnett} in $H^{2}$ by
\begin{equation}
\mathbb{H}(\varphi)f=\mathbb{JP}_{-}\varphi f,\ \ f\in\mathfrak{H}_{2},
\label{H on dense set}%
\end{equation}
and then extend (\ref{H on dense set})\ to the whole $H^{2}$ retaining the
same notation $\mathbb{H}(\varphi)$ for the extension.

Since obviously $\mathbb{H}(\varphi+h)=\mathbb{H}(\varphi)$, for any $h\in
H^{\infty}$, only the part of $\varphi$ analytic in $\mathbb{C}^{-}$ is
essential. We call (see (\ref{eq3.3}))
\begin{equation}
\Phi\overset{\operatorname*{def}}{=}\widetilde{\mathbb{P}}_{-}\varphi
\label{eq4.3}%
\end{equation}
the principal (co-analytic) part of the symbol $\varphi$. It need not be in
$L^{\infty}$ but due to (\ref{BMOA}) $\Phi\in\ $BMOA$\left(  \mathbb{C}%
^{-}\right)  $. The next statement allows us to define Hankel operators with
such symbols.

\begin{theorem}
\label{th4.5}Let $\varphi\in L^{\infty}$ and $h\in\ $BMOA$\left(
\mathbb{C}^{+}\right)  $. Then $\mathbb{H}(\varphi+h)$ is well-defined,
bounded and%
\begin{equation}
\mathbb{H}(\varphi+h)=\mathbb{H}(\varphi). \label{eq4.5'}%
\end{equation}
Consequently%
\begin{equation}
\mathbb{H}(\varphi)=\mathbb{H}(\Phi). \label{eq4.5}%
\end{equation}

\end{theorem}

\begin{proof}
As well-known \cite{Garnett} every BMOA\ function $h$ is subject to $h\left(
x\right)  /\left(  1+x^{2}\right)  \in L^{1}$ and one can easily see that
$hf\in H^{2}$ if $f\in\mathfrak{H}_{2}$. Hence $\mathbb{P}_{-}hf=0$ and
(\ref{eq4.5'}) holds on the set $\mathfrak{H}_{2}$. Therefore (\ref{eq4.5'})
can be closed to the whole $H^{2}$ and $\mathbb{H}(\varphi+h)$ is well-defined
in the sense discussed above, bounded and (\ref{eq4.5'}) holds. By
(\ref{eq3.5}) $\varphi=\widetilde{\mathbb{P}}_{-}\varphi+\widetilde{\mathbb{P}%
}_{+}\varphi$ and (\ref{eq4.5}) follows from (\ref{eq4.5'}) with
$h=-\widetilde{\mathbb{P}}_{+}\varphi$ which, by (\ref{BMOA}), is in
BMOA$\left(  \mathbb{C}^{+}\right)  $.
\end{proof}

{If $\varphi\in$ BMO, then there exist $\varphi_{1},\varphi_{2}\in L^{\infty}$
such that
\begin{equation}
\varphi=\varphi_{1}+\widetilde{\mathbb{P}}\varphi_{2} \label{eq4.11.1}%
\end{equation}
and Theorem \ref{th4.5} immediately implies}

\begin{proposition}
\label{prop4.4.1} Let $\varphi\in$ BMO, then $\mathbb{H}(\varphi)$ is
well-defined and bounded.
\end{proposition}

The final statement of this section trivially follows from (\ref{eq4.7}%
)-(\ref{eq4.9}).

\begin{proposition}
\label{prop4.7} $\mathbb{H}(\varphi)$ is selfadjoint if $\mathbb{J}%
\varphi=\bar{\varphi}$$\;\not .  \;$
\end{proposition}

In the context of integral operators the Hankel operator is usually defined as
an integral operator on $L^{2}(\mathbb{R}_{+})$ whose kernel depends on the
sum of the arguments%
\begin{equation}
(\mathbb{H}f)(x)=\int_{0}^{\infty}h(x+y)f(y)dy,\;f\in L^{2}(\mathbb{R}%
_{+}),\;x\geq0 \label{eq4.10}%
\end{equation}
and it is this form that Hankel operators typically appear in the inverse
scattering formalism. One can show that the Hankel operator $\mathbb{H}$
defined by (\ref{eq4.10}) is unitary equivalent to $\mathbb{H}(\varphi)$ with
the symbol $\varphi$ equal to the Fourier transform of $h$. We emphasize
though that the form (\ref{eq4.10}) does not prove to be convenient for our
purposes and also $h$ is in general not a function but a distribution.

Finally, we also note that $\mathbb{H}(\varphi)$ is unitary equivalent to the
operator $\chi\mathbb{F}^{-1}\varphi\mathbb{F}^{-1}$ on $L^{2}(\mathbb{R}%
_{+})$. Here $\mathbb{F}$ is the Fourier transform. However our previous
experience suggests that this realization of the Hankel operator has some
technical disadvantages to (\ref{eq4.1}).

\section{Hankel operators and the Sarason algebra $H^{\infty}+C$%
\label{H^infty+C}}

The set $H^{\infty}+C$ is one of the most common function classes in the
theory of Hankel and Toeplitz operators. By definition%

\[
H^{\infty}+C\overset{\operatorname*{def}}{=}\{f:f=h+g,\;h\in H^{\infty},\;g\in
C\}.
\]

\begin{theorem}
[Sarason, 1967]\label{sarason thm}$H^{\infty}+C$ is a closed sub-algebra of
$L^{\infty}$.
\end{theorem}

The importance of $H^{\infty}+C$ in the context of Hankel operators is due to
the following fundamental theorem.

\begin{theorem}
[Hartman, 1958]\label{thHart}Let $\varphi\in L^{\infty}$. Then $\mathbb{H}%
(\varphi)$ is compact iff $\varphi\in H^{\infty}+C$. I.e. $\mathbb{H}%
(\varphi)$ is compact iff $\mathbb{H}(\varphi)=\mathbb{H}(g)$ with some $g\in
C$.
\end{theorem}

For Hankel operators appearing in completely integrable systems the membership
of the symbol in $H^{\infty}+C$ is far from being obvious. This may be part of
the reason why the powerful machinery of Hankel operators has not made it to
solution theory. The following statement will be crucial to our approach.

\begin{theorem}
[Grudsky, 2001]\label{thGru01}Let $p(x)$ be a real polynomial with a positive
leading coefficient such that
\begin{equation}
p(-x)=-p(x). \label{eq5.1}%
\end{equation}
Then
\begin{equation}
e^{ip}\in H^{\infty}+C. \label{eq5.2}%
\end{equation}
Moreover, there exist an infinite Blaschke product $B$ and a unimodular
function $u\in C$ such that
\begin{equation}
e^{ip}=Bu. \label{eq5.3}%
\end{equation}

\end{theorem}

It is worth mentioning that this theorem is a particular case of a more
general statement originally obtained in \cite{Gru2001} (see also
\cite{DybGru2002})\ for the case of the unit circle and reformulated for the
real line in \cite{BotGruSpit2001} (see also recent \cite{GruShar2012}). This
statement says that Theorem \ref{thGru01} holds not only for polynomial but
any function $f$ such that
\[
\lim\limits_{x\rightarrow\infty}\inf\dfrac{xf^{\prime\prime}(x)}{f^{\prime
}(x)}>-2,\lim\limits_{x\rightarrow\infty}\dfrac{xf^{\prime\prime}%
(x)}{f^{\prime}\left(  x\right)  ^{2}}=0,\lim\limits_{x\rightarrow\infty
}\dfrac{\sqrt{x}f^{\prime\prime}(x)}{f^{\prime}\left(  x\right)  ^{3/2}}=0.
\]

We emphasize that functions of the form $e^{ip}$ commonly appear in the IST
approach to completely integrable PDEs. For instance, in the KdV case%
\[
p(\lambda)=t\lambda^{3}+x\lambda
\]
with real $x$ (spatial variable) and positive $t$ (time). Note that for
polynomials $p$ of even order, Theorem \ref{thGru01} fails.

\begin{definition}
A function $f\in H^{\infty}+C$ is said invertible in $H^{\infty}+C$ if $1/f\in
H^{\infty}+C$. Similarly, $f$ is not invertible in $H^{\infty}+C$ if
$1/f\notin H^{\infty}+C$.
\end{definition}

This concept is very important in the connection with invertibility of
Toeplitz operators, as the following theorem suggests (see, e.g.
\cite{BotSil06}, \cite{DybGru2002}).

\begin{theorem}
\label{th5.5}Let $\varphi\in H^{\infty}+C$ and $1/\varphi\in L^{\infty}$.
Then
\begin{align}
1/\varphi\notin H^{\infty}+C  &  \Longrightarrow\mathbb{T}(\varphi)\text{ is
left-invertible,}\label{eq5.4}\\
1/\varphi\in H^{\infty}+C  &  \Longrightarrow\mathbb{T}(\varphi)\;\text{is
Fredholm.} \label{eq5.5}%
\end{align}

\end{theorem}

\begin{lemma}
\label{lem5.6}Let $B$ be an infinite Blaschke product, $u\in H^{\infty}+C$ and
unimodular. Then $\varphi=Bu$ is not invertible in $H^{\infty}+C$.
\end{lemma}

\begin{proof}
(By contradiction). Since $B\in H^{\infty}$, due to the algebraic property
(Theorem \ref{sarason thm})\ of $H^{\infty}+C,\;$one has $\varphi\in
H^{\infty}+C$. Assume that $\varphi$ is invertible in $H^{\infty}+C$, i.e.
$1/\varphi\in H^{\infty}+C$. Then by (\ref{eq5.5}) $\mathbb{T}(\varphi)$ is
Fredholm that forces $\mathbb{T}(B)$ to be Fredholm too. Indeed, $B\in
H^{\infty}$ and, since $\varphi=Bu$,
\[
1/B=u\cdot\,1/\varphi\in H^{\infty}+C.
\]
Thus $B$ is invertible in $H^{\infty}+C$ and (\ref{eq5.5}) holds. Hence
$\mathbb{T}(\overline{B})=\mathbb{T}(1/B)$ is also Fredholm and therefore by
definition
\begin{equation}
\dim\ker\mathbb{T}(\overline{B})<\infty. \label{eq5.6}%
\end{equation}
We now show that (\ref{eq5.6}) may not hold for $B$ with infinitely many zeros
$\{z_{k}\}$, which creates a desired contradiction. To this end consider the
Blaschke product (\ref{eq3.8})
\[
B(x)=\prod b_{n}(x),\;b_{n}=c_{n}\left(  \frac{x-z_{n}}{x-\overline{z_{n}}%
}\right)
\]
and set%
\[
f_{n}(x):=c_{n}(x-\overline{z_{n}})^{-1}.
\]
Clearly $f_{n}\in H^{2}$ and
\[
\mathbb{T}(\overline{B})f_{n}=\mathbb{P}_{+}\overline{B}f_{n}=\mathbb{P}%
_{+}\overline{\overline{c_{n}}(\cdot-z_{n})^{-1}B}=\mathbb{P}_{+}(\cdot
-z_{n})^{-1}\overline{B_{n}},
\]
where $B_{n}=B/b_{n}$. But $\overline{B_{n}}\in H_{-}^{\infty}$ and
$(x-z_{n})^{-1}\in H_{-}^{2}$. Hence
\[
(x-z_{n})^{-1}\overline{B_{n}(x)}\in H_{-}^{2}%
\]
and
\[
\mathbb{T}(\overline{B})f_{n}=0.
\]
Therefore $f_{n}\in\ker\mathbb{T}(\overline{B})$ and the lemma is proven as
$\{f_{n}\}$ are linearly independent.
\end{proof}

Note that Lemma \ref{lem5.6} is entirely about $H^{\infty}+C$ but its proof,
as often happens in this circle of issues, relies on operator theoretical arguments.

The next important claim directly follows from Theorem \ref{thGru01} and Lemma
\ref{lem5.6}.

\begin{theorem}
\label{theorem 5.4'}Let $u$ be a unimodular function from $H^{\infty}+C$ and
$e^{ip}$ as in Theorem \ref{thGru01}. Then $e^{ip}u$ is not invertible in
$H^{\infty}+C$.
\end{theorem}

Combining Theorems \ref{th4.4} and \ref{th5.5} yields

\begin{theorem}
\label{th5.7}If $\varphi\in H^{\infty}+C$ and unimodular but not invertible
then
\begin{equation}
\Vert\mathbb{H}(\varphi)\Vert<1. \label{eq5.7}%
\end{equation}

\end{theorem}

\begin{proof}
By Theorem \ref{th5.5}, $\mathbb{T}(\varphi)$ is left-invertible. By Theorem
\ref{th4.4} we have (\ref{eq5.7}).
\end{proof}

While an immediate consequence of Theorems \ref{th5.7} and \ref{thGru01}, the
following theorem is vital to our approach.

\begin{theorem}
\label{Thm 5.6'}If $u\in H^{\infty}+C$, $\left\vert u\right\vert =1$, and $p $
is as in Theorem \ref{thGru01}, then%
\[
\Vert\mathbb{H}(e^{ip}u)\Vert<1.
\]

\end{theorem}

\begin{theorem}
\label{rem5.8} If $\varphi\in H^{\infty}+C$
is not unimodular but $\Vert\varphi\Vert_{\infty}\leq1$ and $\mathbb{J}%
\varphi=\bar{\varphi}$ then (\ref{eq5.7}) holds.
\end{theorem}

\begin{proof}
(By contradiction) Assume that $\Vert\mathbb{H}(\varphi)\Vert=1$. Since
$\mathbb{H}\left(  \varphi\right)  $ is selfadjoint and compact (by
Proposition \ref{prop4.7} and the Hartman theorem respectively),
$\mathbb{H}(\varphi)$ has a unimodular eigenvalue $\lambda$ ($\lambda=\pm1$).
For the associated normalized eigenfunction $f\in H^{2}$ we have by
(\ref{eq4.9})%
\[
\left\langle \mathbb{H}(\varphi)f,f\right\rangle =\left\langle \varphi
f,\mathbb{P}_{-}\mathbb{J}f\right\rangle =\left\langle \varphi f,\mathbb{J}%
f\right\rangle
\]
and hence by the Cauchy inequality%
\begin{align}
\left\vert \left\langle \mathbb{H}(\varphi)f,f\right\rangle \right\vert ^{2}
&  \leq\left(  \int\left\vert \varphi\right\vert \left\vert f\right\vert
\left\vert \mathbb{J}f\right\vert \right)  ^{2}\leq\int\left\vert
\varphi\right\vert \left\vert f\right\vert ^{2}\ \int\left\vert \varphi
\right\vert \left\vert \mathbb{J}f\right\vert ^{2}\nonumber\label{eig}\\
&  \leq\int\left\vert \varphi\right\vert \left\vert f\right\vert ^{2}%
\ \int\left\vert \mathbb{J}f\right\vert ^{2}=\int\left\vert \varphi\right\vert
\left\vert f\right\vert ^{2}\nonumber\\
&  =\int_{S}\left\vert \varphi\right\vert \left\vert f\right\vert ^{2}%
+\int_{\mathbb{R}\diagdown S}\left\vert \varphi\right\vert \left\vert
f\right\vert ^{2}<\Vert f\Vert_{2}^{2}=1,
\end{align}
where $S$ is a set of positive Lebesgue measure where $\left\vert
\varphi\right\vert <1$ a.e. Here we have used the fact that $f\in H^{2}$ and
hence cannot vanish on $S$. The inequality (\ref{eig}) implies that
$\left\vert \lambda\right\vert <1$ which is a contradiction.
\end{proof}

\section{The classical IST and Hankel operators\label{sec6}}

In this section we review some basics of the classical IST and prepare the
necessary bulk of formulas (see, e.g. \cite{Deift79}, \cite{March86}). We will
also demonstrate the convenience of the Hankel operator approach to the Cauchy
problem for the KdV equation in the classical situation of initial data
decaying fast enough. Some derivations are given whenever we have a concise
way to do so.

Through this section we assume that the initial profile $q$ in (\ref{KdV}%
)-(\ref{KdVID}) is real and short range, i.e. $\left(  1+\left\vert
x\right\vert \right)  q\left(  x\right)  \in L^{1}$. In the sequel we refer to
such initial data as classical.

\subsection{Direct scattering problem}

Associate with $q$ the full line Schr\"{o}dinger operator $\mathbb{L}%
_{q}=-\partial_{x}^{2}+q(x)$. As well-known, $\mathbb{L}_{q}$ is self-adjoint
on $L^{2}$ and%
\[
\operatorname*{Spec}(\mathbb{L}_{q})=\{-\kappa_{n}^{2}\}_{n=1}^{N}%
\cup\mathbb{R}_{+}.
\]
The singular spectrum of $\mathbb{L}_{q}$ consists of a finite number of
simple negative eigenvalues $\{-\kappa_{n}^{2}\}$, called bound states, and
absolutely continuous (a.c.) two fold component filling $\mathbb{R}_{+}$.
There is no singular continuous spectrum. Two linearly independent
(generalized) eigenfunctions of the a.c. spectrum $\psi_{\pm}(x,k),\;k\in
\mathbb{R}$, can be chosen to satisfy
\begin{equation}
\psi_{\pm}(x,k)=e^{\pm ikx}+o(1),\;\partial_{x}\psi_{\pm}(x,k)\mp ik\psi_{\pm
}(x,k)=o(1),\ \ x\rightarrow\pm\infty. \label{eq6.2}%
\end{equation}
The functions $\psi_{\pm}$ are referred to as Jost solutions of the
Schr\"{o}dinger equation
\begin{equation}
\mathbb{L}_{q}\psi=k^{2}\psi. \label{eq6.3}%
\end{equation}
We summarize the properties of $\psi_{\pm}$ in

\begin{theorem}
[On Jost solutions]\label{th6.1}The Jost solutions $\psi_{\pm}(x,k)$ are
analytic for $\operatorname{Im}k>0$ and continuous for $\operatorname{Im}%
k\geq0$. Moreover as $k\rightarrow\infty,\;\operatorname{Im}k\geq0$,
\begin{equation}
\psi_{\pm}(x,k)=e^{\pm ikx}\left(  1\pm\frac{i}{2k}\int_{x}^{\pm\infty
}q+O\left(  \frac{1}{k^{2}}\right)  \right)  \label{eq6.4}%
\end{equation}
and
\begin{equation}
\psi_{\pm}(x,-k)=\overline{\psi_{\pm}(x,k)},\;k\in\mathbb{R}. \label{eq6.5}%
\end{equation}

\end{theorem}

This theorem is nearly folklore. One rewrites (\ref{eq6.3}) as $\psi
^{\prime\prime}+k^{2}\psi=q(x)\psi$ and then solves it by variation of
parameters (keeping $q\psi$ as the non homogeneous term) with boundary
conditions (\ref{eq6.2}). The integral equation for $\psi_{\pm}$ obtained this
way is Volterra-type and thus the (necessarily convergent) Neumann series
obtained by iteration readily yields the conclusions of Theorem \ref{th6.1}.

To remove the oscillatory behavior of $\psi_{\pm}$ let us introduce the
functions, sometimes called Faddeev,
\begin{equation}
y_{\pm}(k,x):=e^{\mp ikx}\psi_{\pm}(x,k). \label{y}%
\end{equation}
The function $y:=y_{+}$ will be used more frequently. Its properties
\cite{Deift79} are given in

\begin{theorem}
[On Faddeev functions]\label{Faddeev}For any $x$, the function $y\left(
k,x\right)  $ is analytic in $\mathbb{C}^{+}$, continuous on the real line
and
\begin{equation}
\ y\left(  k,x\right)  \rightarrow1,\ \left\vert k\right\vert \rightarrow
\infty,\ \operatorname{Im}k\geq0. \label{faddeev asympt}%
\end{equation}
All zeros of $y\left(  \cdot,x\right)  ~$in $\mathbb{C}^{+}$ are imaginary and
for their number $N_{x}$ we have%
\[
N_{x}\leq\int_{x}^{\infty}\left(  s-x\right)  \left\vert q(s)\right\vert
ds<\int_{x}^{\infty}s\left\vert q(s)\right\vert ds.
\]
If $k=i\nu$ is a zero of $y\left(  \cdot,x\right)  $ then $k^{2}=-\nu^{2}$ is
a bound state of the Dirichlet Schr\"{o}dinger operator $\mathbb{L}_{q}^{D}$
on $L^{2}(x,\infty)$. The only real zero of $y\left(  \cdot,x\right)  $ could
be $k=0.$ If $y\left(  0,x\right)  =0$ then $k^{2}=0$ is not a bound state of
$\mathbb{L}_{q}^{D}$ on $L^{2}(x,\infty)$. The (full line) operator
$\mathbb{L}_{q}$ has a bound state iff $y\left(  0,x\right)  =0$ for some $x$.
\end{theorem}

Note that the function $\int_{x}^{\infty}\left(  s-x\right)  \left\vert
q(s)\right\vert ds$ is decreasing to zero and hence by Theorem \ref{Faddeev}
$N_{x}=0$ for some $x$. This motivates

\begin{definition}
\bigskip\label{large enough}Let $q$ be subject to Hypothesis \ref{hyp1.1} (2).
We call a number $a$ large enough and denote $a>>1$ if $\int_{a}^{\infty
}\left(  x-a\right)  \left\vert q(x)\right\vert dx<1$.
\end{definition}

We will use Theorem \ref{Faddeev} primarily in the form

\begin{corollary}
\label{Corol of Faddeev}For some $a$ large enough
\begin{equation}
y\left(  \cdot,x\right)  ^{\pm1}\in H^{\infty}\cap C. \label{H infty prop}%
\end{equation}

\end{corollary}

Since $q$ is real, $\overline{\psi_{\pm}}$ also solves (\ref{eq6.3}) and one
can easily see that the pairs $\{\psi_{+},\overline{\psi_{+}}\}$ and
$\{\psi_{-},\overline{\psi_{-}}\}$ form fundamental sets for (\ref{eq6.3}).
Hence $\psi_{\mp}$ is a linear combination of $\{\psi_{\pm},\overline
{\psi_{\pm}}\}$. Elementary Wronskian considerations then yield the (basic)
right/left scattering relations
\begin{equation}
T(k)\psi_{\mp}(x,k)=\overline{\psi_{\pm}(x,k)}+R_{\pm}(k)\psi_{\pm
}(x,k),\;k\in\mathbb{R}, \label{basic scatt identity}%
\end{equation}
where $T,R_{\pm}$, called the transmission and right/left reflection
coefficients respectively. It immediately follows from
(\ref{basic scatt identity}) that $(\operatorname{Im}k=0)$
\begin{align}
T  &  =\frac{2ik}{W(\psi_{-},\psi_{+})},\label{eq6.8}\\
R_{+}  &  =\frac{W(\overline{\psi_{+}},\psi_{-})}{W(\psi_{-},\psi_{+}%
)},\ \ R_{-}=\frac{W(\psi_{+},\overline{\psi_{-}})}{W(\psi_{-},\psi_{+}%
)},\ \label{eq6.10}%
\end{align}
where the Wronskians ($W\left(  f,g\right)  =fg^{\prime}-f^{\prime}g$) are
independent of $x\;$($\partial_{x}\;$is missing in $\mathbb{L}_{q}$). By
Theorem \ref{th6.1}, $W(\psi_{-},\psi_{+})$ is analytic in $\mathbb{C}^{+}$
and by (\ref{eq6.2})%
\[
W(\psi_{-},\psi_{+})=2ik+o(1),\;k\rightarrow\infty,\;\operatorname{Im}k\geq0.
\]
Therefore $T(k)$ is analytic in $\mathbb{C}^{+}$ except for zeros of
$W(\psi_{-},\psi_{+})$ and
\[
T(k)=1+o(1),\;k\rightarrow\infty,\;\operatorname{Im}k\geq0.
\]
If $k_{0}$ is a zero of $W(\psi_{-},\psi_{+})$ then $\psi_{+}(x,k_{0})=\mu
_{0}\psi_{-}(x,k_{0})$ (linearly dependent) with some $\mu_{0}\neq0$, that
occurs only for $k_{0}\in i\mathbb{R}_{+}$ such that $k_{0}^{2}=-\kappa
_{0}^{2}$, where $-\kappa_{0}^{2}$ is a bound state of $\mathbb{L}_{q}$. Next,
from the well-known (and easily verifiable) formula $\partial_{k}W(\psi
_{-},\psi_{+})=2k\int\psi_{-}\psi_{+}$ one has
\[
\left.  \partial_{k}W(\psi_{-},\psi_{+})\right\vert _{k=i\kappa_{0}}%
=2i\kappa_{0}\mu_{0}^{\mp1}\int\psi_{\pm}^{2}(\cdot,i\kappa_{0}),
\]
which means that $i\kappa_{0}$ is a simple zero of $W(\psi_{-},\psi_{+})$. It
follows from (\ref{eq6.8}) that
\begin{align*}
\operatorname*{Res}\limits_{i\kappa_{0}}T  &  =\left.  \frac{2ik}{\partial
_{k}W(\psi_{-},\psi_{+})}\right\vert _{k=i\kappa_{0}}=i\mu_{0}^{\pm1}\left(
\int\psi_{\pm}^{2}(\cdot,i\kappa_{0})\right)  ^{-1}\\
&  =i\mu_{0}^{\pm1}\Vert\psi_{\pm}(\cdot,i\kappa_{0})\Vert_{2}^{-2}.
\end{align*}
The quantity
\begin{equation}
c_{0}^{\pm}\overset{\operatorname*{def}}{=}\Vert\psi_{\pm}(\cdot,i\kappa
_{0})\Vert_{2}^{-2} \label{eq6.11}%
\end{equation}
is called the right/left norming constant of a bound state $-\kappa_{0}^{2}$.
Its role in the IST is fundamental but we are not aware of its clear physical meaning.

The quantities $T,\;R_{\pm},\;(\kappa_{n},c_{n}^{\pm})$ are called scattering
and they can be obtained from the Jost solutions $\psi_{\pm}$ of
$\mathbb{L}_{q}\psi=k^{2}\psi$. In other words, the scattering quantities can
be read off the complete set of the spectral data of $\mathbb{L}_{q}$. We
summarize the information about them in the following theorem.

\begin{theorem}
\label{th6.2}The transmission coefficient $T\in C$ and is analytic in
$\mathbb{C}^{+}$ except for a finite number of simple poles $\{i\kappa
_{n}\}_{n=1}^{N}$ with the residues
\begin{equation}
\operatorname{Res}(T,i\kappa_{n})=i\mu_{n}^{\pm1}c_{n}^{\pm}, \label{eq6.12}%
\end{equation}
where $c_{n}^{\pm}$ are norming constants defined by (\ref{eq6.11}) and
$\mu_{n}$ determined from
\begin{equation}
\psi_{+}(x,i\kappa_{n})=\mu_{n}\psi_{-}(x,i\kappa_{n}). \label{eq6.13}%
\end{equation}
Moreover,%
\[
\lim T(k)=1,\;k\rightarrow\infty,\;\operatorname{Im}k\geq0.
\]
The reflection coefficients $R_{\pm}\in C$ (but need not be analytic),
$|R(k)|<1$ for $k\neq0$ and generically\footnote{I.e. $R\left(  0\right)  >-1$
only in exceptional cases and can be destroyed by a small perturbation.}
$R(0)=-1$. Furthermore,%
\begin{equation}
T(-k)=\overline{T(k)},\ \ \ R_{\pm}(-k)=\overline{R_{\pm}(k)},\ \ \ \left\vert
T\left(  k\right)  \right\vert ^{2}+\left\vert R\left(  k\right)  \right\vert
^{2}=1,\ \ k\in\mathbb{R}. \label{eq6.14}%
\end{equation}

\end{theorem}

\subsection{Inverse scattering problem\label{inverse scat}}

One asks what is the minimal subset of scattering quantities that determines
$q$ completely?

The answer can be seen from the following arguments. Take one (e.g. right) of
the basic scattering relations (\ref{basic scatt identity}) and rewrite it in
the form (recall (\ref{cubic exp}))
\begin{equation}
Ty_{-}=\bar{y}_{+}+R\xi_{x}y_{+} \label{eq6.15}%
\end{equation}
Let us regard (\ref{eq6.15}) as a Hilbert-Riemann problem of determining
$y_{\pm}$ by given $T,R$ which we will solve by Hankel operator techniques.
The potential $q$ can then be easily found by (\ref{eq6.4}).

By Theorems \ref{th6.1}, \ref{th6.2} $Ty_{-}$ in (\ref{eq6.15}) is meromorphic
in $\mathbb{C}^{+}$ with simple poles at $i\kappa_{n}$ and residences%
\begin{align}
\operatorname*{Res}\limits_{k=i\kappa_{n}}T\left(  k\right)  y_{-}\left(
k,x\right)   &  =y_{-}(i\kappa_{n},x)\operatorname*{Res}\limits_{k=i\kappa
_{n}}T\left(  k\right) \nonumber\label{res}\\
&  =i\mu_{n}y_{-}(i\kappa_{n},x)c_{n}^{+}=ic_{n}^{+}\xi_{x}(i\kappa
_{n})y(i\kappa_{n},x),
\end{align}
where we have used (\ref{eq6.12}), (\ref{eq6.13}). Note now that for each
fixed $x$
\[
T\left(  k\right)  y_{-}\left(  k,x\right)  -1-\sum_{n=1}^{N}\frac{ic_{n}%
\xi_{x}(i\kappa_{n})}{k-i\kappa_{n}}y(i\kappa_{n},x)\in H^{2}.
\]
Abbreviating $R_{x}:=R\xi_{x},\ c_{x,n}:=c_{n}^{+}\xi_{x}(i\kappa_{n}),$
rewrite (\ref{eq6.15}) in the form%

\begin{align}
&  T\left(  k\right)  y_{-}\left(  k,x\right)  -1-%
{\displaystyle\sum\limits_{n=1}^{N}}
\frac{ic_{x,n}}{k-i\kappa_{n}}y\left(  i\kappa_{n},x\right) \nonumber\\
&  =\overline{\left(  y\left(  k,x\right)  -1\right)  }+R_{x}\left(  k\right)
\left(  y\left(  k,x\right)  -1\right) \nonumber\\
&  +R_{x}\left(  k\right)  -%
{\displaystyle\sum\limits_{n=1}^{N}}
\frac{ic_{x,n}}{k-i\kappa_{n}}y\left(  i\kappa_{n},x\right)  . \label{eq6.16}%
\end{align}
Noticing that the last term in (\ref{eq6.16}) is in $H_{-}^{2}$, we can apply
the Riesz projection $\mathbb{P}_{-}$ to (\ref{eq6.16}). Thus
\begin{equation}
\mathbb{P}_{-}(\overline{Y}+R_{x}Y)+\mathbb{P}_{-}R_{x}-\sum_{n=1}^{N}%
ic_{x,n}\ \frac{Y\left(  i\kappa_{n},x\right)  }{\cdot-i\kappa_{n}}-\sum
_{n=1}^{N}\frac{ic_{x,n}}{\cdot-i\kappa_{n}}=0, \label{eq6.17}%
\end{equation}
where $Y:=y-1$. It is clear that $Y\in H^{2}$ for any $x\in\mathbb{R}$. Due to
(\ref{eq6.5}), $\overline{Y}=\mathbb{J}Y$ and by (\ref{eq4.9}) we have
\begin{equation}
\mathbb{P}_{-}\overline{Y}=\mathbb{P}_{-}\mathbb{J}Y=\mathbb{JP}%
_{+}Y=\mathbb{J}Y. \label{eq6.18}%
\end{equation}
By (\ref{P_})
\begin{equation}
\sum_{n=1}^{N}ic_{x,n}\ \frac{Y\left(  i\kappa_{n},x\right)  }{\cdot
-i\kappa_{n}}=\mathbb{P}_{-}\sum_{n=1}^{N}ic_{x,n}\ \frac{Y(\cdot,x)}%
{\cdot-i\kappa_{n}}. \label{eq6.19}%
\end{equation}
Inserting (\ref{eq6.18}) and (\ref{eq6.19}) into (\ref{eq6.17}), we obtain
\[
\mathbb{J}Y+\mathbb{P}_{-}\left(  R_{x}-\sum_{n=1}^{N}\frac{ic_{x,n}}%
{\cdot-i\kappa_{n}}\right)  Y=-\mathbb{P}_{-}\left(  R_{x}-\sum_{n=1}^{N}%
\frac{ic_{x,n}}{\cdot-i\kappa_{n}}\right)  .
\]
Applying $\mathbb{J}$ to both sides of this equation yields
\begin{equation}
(\mathbb{I}+\mathbb{H}(\varphi))Y=-\mathbb{H}(\varphi)1, \label{eq6.20}%
\end{equation}
where $\mathbb{H}(\varphi)$ is the Hankel operator defined in Definition
\ref{def4.1} with symbol%
\[
\varphi\left(  k\right)  =\varphi_{x}(k)=R(k)\xi_{x}(k)+\sum_{n=1}^{N}%
\frac{c_{n}\xi_{x}(i\kappa_{n})}{\kappa_{n}+ik}%
\]
where $x$ is a real parameter $\left(  \xi_{x}(k)=e^{2ikx}\right)  $.

Due to (\ref{eq6.14}), $\mathbb{J}\varphi=\overline{\varphi}$ and hence by
Proposition \ref{prop4.7} $\mathbb{H}(\varphi)$ is selfadjoint. Note that
$\mathbb{H}(\varphi)1$ on the right hand side of (\ref{eq6.20}) should be
interpreted as
\[
\mathbb{H}(\varphi)1=\mathbb{P}_{+}\bar{\varphi}\in H^{2}.
\]

It is now clear that if we show that (\ref{eq6.20}) is uniquely solvable and
$Y(x,k)$ is its solution then the potential $q\left(  x\right)  $ can be found
from (\ref{eq6.4}) by
\begin{equation}
q(x)=\partial_{x}\lim2ikY(k,x),\ \ \ k\rightarrow\infty. \label{eq6.22}%
\end{equation}

Thus, (\ref{eq6.20}) suggests that what one needs to know to recover $q$ is
the (right) reflection coefficient $R(k)$ for $k\geq0$, bound states
$\{-\kappa_{n}^{2}\}_{n=1}^{N}$ and their (right) norming constants
$\{c_{n}\}_{n=1}^{N}$\footnote{Similarly, the left reflection coefficient
$R_{-}$ and left norming constants $\{c_{n}^{-}\}$ in place of $R,c_{n}$.}.

The set $S_{q}=\{R(k),\;k\geq0,\;(\kappa_{n},c_{n})_{n=1}^{N}\}$ is called the
(right)\ scattering data for $\mathbb{L}_{q}$. The solubility of
(\ref{eq6.20}) is equivalent to bounded invertibility of $\mathbb{I}%
+\mathbb{H}(\varphi)$.

\subsection{Inverse scattering transform}

To reformulate the classical IST in terms of Hankel operators, we recall the
classical fact that the initial short range profile $q$ in (\ref{KdV}%
)-(\ref{KdVID}) evolves under the KdV flow in such a way that the scattering
data $S_{q}(t)$ for $q(x,t)$ evolves by (\ref{time evol}). It is convenient to
introduce
\[
S_{q}(x,t)\overset{\operatorname*{def}}{=}\left\{  R(k)\xi_{x,t}%
(k),\;k\geq0,\;(\kappa_{n},c_{n}\xi_{x,t}(i\kappa_{n}))_{n=1}^{N}\right\}  ,
\]
the time evolved scattering data corresponding to the shifted initial profile
$q(\cdot+x)$.

Observe that the KdV flow preserves at least the Schwartz class (an elementary
well-known fact based on pure PDE techniques) and hence the inverse scattering
procedure discussed in Subsection \ref{inverse scat} also applies to the
scattering data $S_{q}(x,t)$. It is remarkable that if one solves
(\ref{eq6.20}) with
\begin{equation}
\varphi=\varphi_{x,t}(k)=R(k)\xi_{x,t}(k)+\sum_{n=1}^{N}\frac{c_{n}\xi
_{x,t}(i\kappa_{n})}{\kappa_{n}+ik} \label{eq6.23}%
\end{equation}
by the Fredholm series formula then $q(x,t)$ computed by (\ref{eq6.22})
simplifies to
\begin{equation}
q(x,t)=-2\partial_{x}^{2}\log\det(\mathbb{I}+\mathbb{H}(\varphi_{x,t})),
\label{eq6.24}%
\end{equation}
where the determinant is understood in the classical Fredholm sense.

The formula (\ref{eq6.24}) is a derivation of the well-known Dyson (also
called Bargman or log-determinant) formula (see, e.g. \cite{Dyson76},
\cite{Popper84}).

We finally arrive at the following version of the classical IST%
\begin{equation}
q(x)\overset{\text{(\ref{eq6.23})}}{\longrightarrow}\mathbb{H}(\varphi
_{x,t})\overset{\text{(\ref{eq6.24})}}{\longrightarrow}q(x,t). \label{eq6.25}%
\end{equation}
There has been nothing new in this section. Our derivation of (\ref{eq6.20})
(its Fourier representation in the famous Gelfand-Levitan-Marchenko integral
equation) consists of well-known classical components \cite{NovikovetalBook}.
Even if (\ref{eq6.25}) has not explicitly appeared in the literature before,
it does not add much value to the classical IST. We shall use, however,
(\ref{eq6.25}) as a suitable starting point to extend IST far beyond standard
assumption (such as decay at infinity) on the initial data $q$. But much
deeper understanding of the Hankel operator is required.

\section{The structure of the reflection coefficient\label{Refl}}

To continue our program we shall understand the structure of reflection
coefficient $R$ appearing in (\ref{eq6.23}). We treat first the classical case
and it will then be quite transparent how to generalize it.

\subsection{The classical case}

We consider the right reflection coefficient $R:=R_{+}$ only. Recall our
notation (\ref{cubic exp}).

\begin{proposition}
[Structure of the classical reflection coefficient]\label{prop7.2}Suppose $q$
is real and such that $\left(  1+\left\vert x\right\vert \right)  q\left(
x\right)  \in L^{1}$. Let $\{R,(\kappa_{n},c_{n})\}$ denote the scattering
data. Then for some $a$ large enough (in the sense of Definition
\ref{large enough}) the (right) reflection coefficient $R$ can be split into
\begin{equation}
R=A_{a}+r_{a}\xi_{a}^{-1}. \label{R-split}%
\end{equation}
The function $A_{a}$ is meromorphic in $\mathbb{C}^{+}$ with the simple
poles\footnote{Recall $-\kappa_{n}^{2}\in\operatorname{Spec}(\mathbb{L}%
_{q}),\;n=1,2,\ldots,N.$} $\{i\kappa_{n}\}_{n=1}^{N}$ and corresponding
residues
\begin{equation}
\operatorname{Res}\left(  A_{a},i\kappa_{n}\right)  =ic_{n} \label{residues}%
\end{equation}
and admits the representations
\begin{align}
A_{a}\left(  k\right)   &  =T\left(  k\right)  \frac{\psi_{-}\left(
a,k\right)  }{\psi_{+}\left(  a,k\right)  }\label{R rep 1}\\
&  =\xi_{a}^{-1}\left(  k\right)  S_{a}\left(  k\right)  /B\left(  k\right)  ,
\label{R rep 2}%
\end{align}
where $S_{a}\in H^{\infty}\cap C,\ \left\Vert S_{a}\right\Vert _{\infty}%
\leq2,~\ $and
\begin{equation}
B\left(  k\right)  =\prod_{n=1}^{N}\frac{k-i\kappa_{n}}{k+i\kappa_{n}},
\label{B class}%
\end{equation}
is the (finite) Blaschke product with simple zeros at $\{i\kappa_{n}%
\}_{n=1}^{N}$. For $r_{a}$ we have
\begin{equation}
r_{a}(k)=-\overline{y\left(  k,a\right)  }/y\left(  k,a\right)  \in C.
\label{eq7.9}%
\end{equation}

\end{proposition}

\begin{proof}
From the right basic scattering relation (\ref{basic scatt identity}) one has
\begin{equation}
R(k)=T(k)\frac{\psi_{-}\left(  k,a\right)  }{\psi_{+}\left(  k,a\right)
}-\frac{\overline{\psi_{+}\left(  k,a\right)  }}{\psi_{+}\left(  k,a\right)  }
\label{R1}%
\end{equation}
and by (\ref{y}) equation (\ref{R-split}) follows with $A_{a}$ and $r_{a}$
given by (\ref{R rep 1}) and (\ref{R rep 2}) respectively. By Theorems
\ref{th6.1} and \ref{th6.2}, $\psi_{\pm}$ and $T$ are analytic and hence
$A_{a}$ is meromorphic in $\mathbb{C}^{+}$. Next, by Corollary
\ref{Corol of Faddeev} a number $a$ can be found so that$\ A_{a}\left(
k\right)  $ and $T\left(  k\right)  \psi_{-}\left(  k,a\right)  $ share the
same poles. For the residues, by Theorem \ref{th6.2}, one has%
\[
\operatorname*{Res}_{k=i\kappa_{n}}T(k)\frac{\psi_{-}(a,k)}{\psi_{+}%
(a,k)}=\frac{\psi_{-}(a,i\kappa_{n})}{\psi_{+}(a,i\kappa_{n})}%
\operatorname*{Res}_{k=i\kappa_{n}}T(k)=\frac{1}{\mu_{n}}i\mu_{n}c_{n}%
=ic_{n}.
\]
Since $c_{n}>0$, $A_{a}\left(  k\right)  $ and $T\left(  k\right)  $ also
share same poles.

Let us show (\ref{R rep 2}). It is well-known \cite{Deift79} that
\[
T(k)=S\left(  k\right)  /B\left(  k\right)  ,
\]
where $B$ is given by (\ref{B class}) and
\[
S\left(  k\right)  =\exp\left\{  \frac{1}{2\pi i}\int\frac{\log\left(
1-\left\vert R\left(  s\right)  \right\vert ^{2}\right)  }{s-k}dk\right\}  .
\]
One can now easily see from Theorem \ref{th6.2} that $S\in H^{\infty}\cap
C~$(even an outer function), and $\left\Vert S\right\Vert _{\infty}\leq1.$ The
representation (\ref{R rep 2}) follows from (\ref{R rep 1}) with $S_{a}\left(
k\right)  =S\left(  k\right)  y_{-}\left(  k,a\right)  /y_{+}\left(
k,a\right)  $. By Corollary \ref{Corol of Faddeev}, $S_{a}\in H^{\infty}$. It
remains to estimate its $H^{\infty}$-norm. Due to (\ref{H^p norm}) we can do
it on the real line. By (\ref{R rep 2}) and (\ref{R-split})
\begin{align*}
\left\vert S_{a}\right\vert  &  =\left\vert \xi_{a}BA_{a}\right\vert
=\left\vert \xi_{a}B\left(  R-r_{a}\xi_{a}^{-1}\right)  \right\vert \\
&  \leq\left\vert R-r_{a}\xi_{a}^{-1}\right\vert \leq2.
\end{align*}
Noticing that by (\ref{H infty prop}) $r_{a}\in C$ concludes the proof.
\end{proof}

\bigskip Note that it is claimed in \cite{McLeodOlver83} (but no rigorous
arguments are provided) that $R(k)$ can be analytically continued into the
upper half plane under the only assumption that $\left(  1+x^{2}\right)  q\in
L^{1}$. If it was true then $\psi_{+}\left(  \cdot,a\right)  $ would
analytically continue into $\mathbb{C}_{-}$. The latter requires an
exponential decay of $q$.

The remarkable feature of (\ref{R-split}) is that while $A_{a}\left(
k\right)  $ does depend on $a$ but for $a>>1$ the poles of $A_{a}\left(
k\right)  $ occur only at the purely imaginary points $\{i\kappa_{n}%
\}_{n=1}^{N}$ such that $\{-\kappa_{n}^{2}\}_{n=1}^{N}$ is the set of bound
states of $\mathbb{L}_{a}$, the residues of $A_{a}$ at $i\kappa_{n}$ being the
norming constant $c_{n}$. This means that for $a>>1$ the function $A_{a}$
uniquely recovers the bound state information $(\kappa_{n},c_{n})$ and hence
the knowledge of $\{R,A_{a}\}$ is equivalent to the knowledge of
$\{R,(\kappa_{n},c_{n})\}$. Thus if we know (say) $\left.  q\right\vert
_{\mathbb{R}_{+}}$ we can find $y\left(  k,a\right)  $ for $a>>1$ and hence
$r_{a}\left(  k\right)  $ for any real $k$. One then computes $A_{a}$ by
(\ref{R-split}). We state what we have arrived at in two corollaries.

\begin{corollary}
\label{finite rou}The measure
\[
d\rho\left(  s\right)  =-i\sum_{n=1}^{N}\operatorname*{Res}\left(
A_{a},i\kappa_{n}\right)  \delta\left(  s-\kappa_{n}\right)  ds=\sum_{n=1}%
^{N}c_{n}\delta\left(  s-\kappa_{n}\right)  ds
\]
is independent of $a>>1$.
\end{corollary}

\begin{corollary}
\label{Partion info}For $a>>1~\ $the pair $\{R,A_{a}\}$ is a set of scattering
data, i.e. it recovers the potential $q$ uniquely. Moreover, if $R$ and
$\left.  q\right\vert _{\mathbb{R}_{+}}$ are known then $\left.  q\right\vert
_{\mathbb{R}_{-}}$ is also known.
\end{corollary}

\begin{remark}
Proposition \ref{prop7.2} is totally elementary but will nevertheless play a
principal role in our considerations. It is worth mentioning that Corollary
\ref{Partion info} (which we don't actually use) immediately implies many
relevant results of \cite{Akt96}, \cite{AKM93}, \cite{BSL95}, \cite{GW95},
\cite{RS94} on the so-called inverse problems with partial information on the
potential (see also \cite{GS97} in this context).
\end{remark}

\subsection{The general step-like case}

Through this subsection we deal with potentials subject to Hypothesis
\ref{hyp1.1}. We start with a brief review of Titchmarsh-Weyl theory of order
two differential operators in dimension one (see, e.g. \cite{TeschlBOOK}).

A real-valued locally integrable potential $q$ is said to be Weyl limit point
at $\pm\infty$ if the equation $\mathbb{L}_{q}u=\lambda u$ has a unique (up to
a multiplicative constant) solution\footnote{Note that Weyl solutions depend
on the spectral parameter (energy) $\lambda$ while Jost solutions are
typically considered as dependent on momentum $\sqrt{\lambda}.$} $\Psi_{\pm
}(\cdot,\lambda)\in L^{2}(a,\pm\infty)$ for each $\lambda\in\mathbb{C}^{+}$.
Such $\Psi_{\pm}$ is commonly called the Weyl solution on $(a,\pm\infty)$. The
existence of $\Psi_{\pm}$ is directly related to the selfadjointness of
$\mathbb{L}_{q}$ on $L^{2}(a,\pm\infty)$ with a Dirichlet (or any other
selfadjoint) condition at $x=a\pm0$. If $q\in L^{1}$ then the Weyl solutions
$\Psi_{\pm}(x,\lambda)\sim e^{\pm i\sqrt{\lambda}x},\;x\rightarrow\pm\infty$,
clearly turn into Jost and we have
\[
\Psi_{\pm}\left(  x,k^{2}\right)  =\psi_{\pm}\left(  x,k\right)  .
\]
However Weyl solutions exist under much more general conditions on $q$'s and
no decay of any kind is required. There is no criterion for the limit point
case in terms of $q$ (a major unsolved problem) but there are a number of
sufficient conditions which are typically satisfied in most of realistic
situations. For instance, any $q$ subject to Hypothesis \ref{hyp1.1} is in the
limit point case at $\pm\infty$.

The following concept is fundamental in spectral theory of ordinary
differential operators.

\begin{definition}
[$m$-function]\label{def1} The\ function%
\begin{equation}
m_{\pm}\left(  \lambda,x\right)  =\pm\frac{\partial_{x}\Psi_{\pm}\left(
x,\lambda\right)  }{\Psi_{\pm}\left(  x,\lambda\right)  },\ \ \lambda
\in\mathbb{C}_{+} \label{m-funct}%
\end{equation}
is called the (Dirichlet, principal) Titchmarsh-Weyl $m-$function, or just $m
$-function.
\end{definition}

By definition $m_{\pm}\left(  \lambda,x\right)  $ depends on two variables
$\left(  \lambda,x\right)  $. The first one, energy, is the main variable. The
other one is typically set $x=0$ with the convention $m_{\pm}\left(
\lambda,0\right)  =m_{\pm}\left(  \lambda\right)  $. It is well-known that
$m_{\pm}\left(  \cdot,x\right)  $ is a Herglotz function. That is, it is
analytic and maps $\mathbb{C}^{+}$ to $\mathbb{C}^{+}$. The following general
statement \cite{ArDon1957} is frequently used in spectral theory of ordinary
differential operators.

\begin{theorem}
[Aronszajn-Donoghue, 1957]\label{Heglotz rep}Let \bigskip$f(\lambda)$ be a
Herglotz function. Then there exists a non-negative measure $d\mu$ such that
\begin{equation}
f(\lambda)=a+b\lambda+\int\frac{1+\lambda s}{s-\lambda}\frac{d\mu(s)}{1+s^{2}%
}, \label{Herglotz}%
\end{equation}
where
\[
a=\operatorname{Re}f\left(  i\right)  ,\ b\geq0,\
{\displaystyle\int}
\dfrac{d\mu(s)}{1+s^{2}}<\infty.
\]
Moreover, $\mu\,\ $is computed by the Herglotz inversion formula%
\begin{equation}
\mu(\Delta)=\lim_{\varepsilon\rightarrow0+}\frac{1}{\pi}\int_{\Delta
}\operatorname{Im}f\left(  s+i\varepsilon\right)  ds.
\label{Herglotz inversion}%
\end{equation}

\end{theorem}

The formula (\ref{Herglotz}) is called the Herglotz or Riesz-Herglotz
representation. It is straightforward to derive from Theorem \ref{Heglotz rep}
the following

\begin{corollary}
\label{Hergtotz rep 2}If \bigskip$f(\lambda)$ is Herglotz and $f\left(
\lambda\right)  \rightarrow0$\ as $\lambda\rightarrow\infty$ along any ray
$0<\varepsilon<\arg\lambda<\pi-\varepsilon,$ and the support of $\mu$ in
(\ref{Herglotz}) is bounded from below then (\ref{Herglotz}) reads%
\begin{equation}
f(\lambda)=\int\frac{d\mu(s)}{s-\lambda}, \label{Herglotz 2}%
\end{equation}
where the measure $\mu$ is subject to $%
{\displaystyle\int}
\dfrac{\left\vert s\right\vert \ d\mu(s)}{1+s^{2}}<\infty.$
\end{corollary}

\bigskip Titchmarsh-Weyl $m-$functions have many important properties. E.g.
the classical Borg-Marchenko result says that the $m-$function determines the
potential uniquely. This however is immaterial to us as apposed to the
following convergence property which we state only for $m=m_{+}$.

\begin{proposition}
\label{convergence property} Let $q,q_{n}$ be in the limit point case at
$+\infty$ and suppose that $q_{n}\rightarrow q$ in $L_{\operatorname*{loc}%
}^{1}$. I.e. for any finite interval $I$
\[
\int_{I}\left\vert q-q_{n}\right\vert \rightarrow0,\ n\rightarrow\infty.
\]
Then
\[
m_{n}\rightrightarrows m,\ \ \ n\rightarrow\infty,\text{ in }\mathbb{C}^{+}%
\]
and hence for the associated spectral measures of the half-line Dirichlet
Schr\"{o}dinger operators one has
\[
\mu_{n}\rightarrow\mu,\ \ \ n\rightarrow\infty,\text{ \ \ weakly.}%
\]

\end{proposition}

We are not sure whom to attribute this statement. It appears in
\cite{CarLac90} as a part of a lemma\footnote{Stated there for
$L_{\operatorname*{loc}}^{2}$ as it was enough for the future purposes. The
actual proof needs $L_{\operatorname*{loc}}^{1}$.}. We learned it first from
\cite{PavSmirn82} but it may have been known much earlier as its proof rests
on original ideas behind the limit point/limit circle classification.

The main convenience of $m$-function in our setting is that classical
scattering theory can be extended far beyond strong decay assumptions at
$\pm\infty$ if Jost solutions are suitably replaced with Weyl \cite{GNP97}.
For instance, one can formally define transmission and reflection coefficients
merely by (\ref{eq6.8})-(\ref{eq6.10}). Such generalizations, however, need
not have suitable properties which could be a real problem.

Let us introduce now the right reflection coefficient for potentials subject
to Hypothesis \ref{hyp1.1}. Since $W\left(  \psi_{+},\overline{\psi_{+}%
}\right)  =-2ik$ the pair $\{\psi_{+},\overline{\psi_{+}}\}$ forms a
fundamental set for $\mathbb{L}_{q}u=k^{2}u$ and hence the Weyl solution
$\Psi_{-}$ is a linear combination of $\{\psi_{+},\overline{\psi_{+}}\}$. I.e.
for any real $k\neq0$%
\begin{equation}
T(k)\Psi_{-}(x,k^{2})=\overline{\psi_{+}(x,k)}+R(k)\psi_{+}(x,k),
\label{eq8.1}%
\end{equation}
holds with some $T$ and $R$. In analogy with (\ref{basic scatt identity}) we
call (\ref{eq8.1}) the (right) basic scattering relation and similarly to
(\ref{eq6.8}) we introduce

\begin{definition}
[Reflection coefficient]\label{def R}We call
\begin{equation}
R\left(  k\right)  =\frac{W(\overline{\psi_{+}}\left(  \cdot,k\right)
,\Psi_{-}\left(  \cdot,k^{2}\right)  )}{W(\Psi_{-}\left(  \cdot,k^{2}\right)
,\psi_{+}\left(  \cdot,k\right)  )} \label{eq8.2}%
\end{equation}
the (right) reflection coefficient.
\end{definition}

Observe that since $W(\Psi_{-},\psi_{+})$ is analytic in $\mathbb{C}^{+}$ away
from $k^{2}\in\operatorname*{Spec}(\mathbb{L}_{q})\cap\mathbb{R}_{-}$, the
denominator $W(\Psi_{-},\psi_{+})$ in (\ref{eq8.2}) cannot vanish on a set of
positive Lebesgue measure. Therefore $R$ is well defined by (\ref{eq8.2}) for
a.e. $k\in\mathbb{R}$. Similarly, $T$ is also well-defined by
\[
T\left(  k\right)  =\frac{2ik}{W(\Psi_{-}\left(  \cdot,k^{2}\right)  ,\psi
_{+}\left(  \cdot,k\right)  )}.
\]

\begin{proposition}
[Properties of the reflection coefficient]\label{props of R}The reflection
coefficient $R$ is symmetric $R\left(  -k\right)  =\overline{R\left(
k\right)  }$ and contractive $\left\vert R\left(  k\right)  \right\vert \leq1$
a.e. Moreover, if $\sigma\left(  \mathbb{L}_{q}\right)  $ is the minimal
support of the two fold a.c. spectrum of $\mathbb{L}_{q}$ then $\left\vert
R\left(  k\right)  \right\vert <1$ for a.e. real $k$ such that $k^{2}\in$
$\sigma\left(  \mathbb{L}_{q}\right)  $ and $\left\vert R(k)\right\vert =1$ otherwise.
\end{proposition}

\begin{proof}
From (\ref{eq8.1})-(\ref{eq8.1}) and the Wronskian identity
\[
W_{12}W_{34}+W_{13}W_{42}+W_{14}W_{23}=0,\;W_{ik}:=W(f_{i},f_{k}),
\]
omitting a straightforward computation, we have%
\begin{equation}
\left\vert R(k)\right\vert ^{2}+\frac{\left\vert T(k)\right\vert ^{2}}%
{2ik}W\left(  \Psi_{-},\overline{\Psi_{-}}\right)  \left(  x,k^{2}+i0\right)
=1 \label{eq8.4}%
\end{equation}
for any $x\in\mathbb{R}$ and a.e. $k\in\mathbb{R}$ or equivalently
\begin{equation}
\left\vert R(k)\right\vert ^{2}+\frac{\operatorname{Im}m_{-}(k^{2}%
+i0,x)}{\left\vert k\right\vert }\cdot\left\vert T(k)\Psi_{-}(x,k^{2}%
+i0)\right\vert ^{2}=1. \label{eq8.5}%
\end{equation}
It remains to notice that $\sigma\left(  \mathbb{L}_{q}\right)  $ coincides
with the closure of $\left\{  \operatorname{Im}m_{-}(t+i0,x)>0\right\}  $.
\end{proof}

If $\Psi_{-}$ is the Jost solution (e.g. the classical case) then $W\left(
\psi_{-}(\cdot,k),\overline{\psi_{-}(\cdot,k)}\right)  =2ik$ and (\ref{eq8.4})
turns into $\left\vert R\left(  k\right)  \right\vert ^{2}+\left\vert T\left(
k\right)  \right\vert ^{2}=1$ as one would expect.

Consider the important case of $q(x)\rightarrow-h^{2},\;x\rightarrow-\infty$,
sufficiently fast. Then for any $k\in\mathbb{R}$%
\[
\Psi_{-}(x,k)=e^{-i\sqrt{k^{2}+h^{2}}x}+o(1),\;x\rightarrow-\infty,
\]

\[
\partial_{x}\Psi_{-}(x,k)+i\sqrt{k^{2}+h^{2}}\,\Psi_{-}%
(x,k)=o(1),\;x\rightarrow-\infty,
\]
and (\ref{eq8.4}) implies
\[
\left\vert R(k)\right\vert ^{2}+\sqrt{1+k^{2}/h^{2}}\left\vert T(k)\right\vert
^{2}=1.
\]

\begin{remark}
As opposed to the classical case, $\{\Psi_{-},\overline{\Psi_{-}}\}$ could be
linearly dependent. In fact, $\{\Psi_{-},\overline{\Psi_{-}}\}$ are linearly
dependent on the support of the set $\operatorname{Im}m_{-}(k^{2}+i0,x)=0$.
The latter may occur, e.g., if $q$ approaches $+\infty$ at $-\infty$ or for
bounded $q$ without a specific pattern of behavior at $-\infty$ (e.g. the
Gaussian white noise). The left basic scattering identity is then undefined
for any $k$ but the right one remains defined. For a fairly complete
description of different spectral regimes we refer to \cite{GNP97}.
\end{remark}

Here is the main statement of this section, which will be crucially used in
the analysis of our Hankel operator.

\begin{proposition}
[Analytic split formula]\label{step like refl}For some $a>>1$ (in the sense of
Definition \ref{large enough}) the reflection coefficient can be represented
as%
\begin{equation}
R=A_{a}\left(  k\right)  +r_{a}\left(  k\right)  \xi_{a}\left(  k\right)
^{-1},\ \operatorname{Im}k=0, \label{split step}%
\end{equation}
where
\begin{equation}
r_{a}\left(  k\right)  =-\overline{y\left(  k,a\right)  }/y\left(  k,a\right)
\in C, \label{r_a}%
\end{equation}
and ($\psi_{+}=:\psi$)
\begin{equation}
A_{a}\left(  k\right)  =\frac{1}{\psi\left(  a,k\right)  ^{2}}\ \frac
{2ik}{m_{+}(k^{2},a)+m_{-}(k^{2},a)}. \label{A1}%
\end{equation}
The function $A_{a}$ is analytic in $\mathbb{C}^{+}$ except for%
\[
i\Delta=\left\{  k\in i\mathbb{R}_{+}\colon k^{2}\in
\text{$\operatorname*{Spec}$}(\mathbb{L}_{q})\cap\mathbb{R}_{-}\right\}
\]
and
\begin{equation}
\left\vert A_{a}\left(  k\right)  \right\vert \leq2\text{ for a.e. }%
k\in\mathbb{R}. \label{bound on A}%
\end{equation}
Furthermore, for the jump $A_{a}(is-0)-A_{a}(is+0)$ across $i\Delta$ we have%
\begin{align}
&  i\left(  A_{a}(is-0)-A_{a}(is+0)\right)  ds/2\pi=\psi(a,is)^{-2}d\mu
_{a}(-s^{2})\label{d rou}\\
&  =:\ d\rho\left(  s\right)  ,\nonumber
\end{align}
where%
\begin{equation}
\ d\mu_{a}(\lambda)=-\frac{1}{\pi}\operatorname{Im}\left[  m_{+}%
(\lambda+i0,a)+m_{-}(\lambda+i0,a)\right]  ^{-1}d\lambda. \label{dmu}%
\end{equation}
The measure $\ d\rho$ is non-negative, finite, supported on $\Delta$, and
independent of $a$.
\end{proposition}

\begin{proof}
The split (\ref{split step}) is obtained same way as (\ref{R-split}). By
(\ref{eq8.2}) and (\ref{m-funct}) we have
\[
R\left(  k\right)  =\frac{1}{\psi\left(  a,k\right)  ^{2}}\ \frac{2ik}%
{m_{+}(k^{2},a)+m_{-}(k^{2},a)}-\frac{\overline{\psi\left(  a,k\right)  }%
}{\psi\left(  a,k\right)  }\xi_{a}^{-1}\left(  k\right)
\]
which proves (\ref{split step}) with $A_{a}$ given by (\ref{A1}). The bound
(\ref{bound on A}) follows from Proposition \ref{props of R} and the obvious
fact $\left\vert r_{a}\xi_{a}{}^{-1}\right\vert =1$.

Note that since $m_{\pm}$ are both Herglotz, the function
\[
f_{a}(\lambda)=-\left(  m_{+}(\lambda,a)+m_{-}(\lambda,a)\right)  ^{-1}%
\]
is also Herglotz. It follows from (\ref{Herglotz inversion}) that its
representing measure $d\mu_{a}$\footnote{Through the paper we use the
convention%
\[
\operatorname{Im}f\left(  t+i0\right)  dt:=w-\lim_{\varepsilon\rightarrow
+0}\operatorname{Im}f\left(  t+i\varepsilon\right)  dt.
\]
}, given by (\ref{dmu}), is non-negative, finite ($\int\frac{d\mu}{1+t^{2}%
}<\infty$) and supported \cite{TeschlBOOK} on the spectrum of $\mathbb{L}_{q}%
$. Now, from (\ref{A1}) and (\ref{dmu})
\begin{align*}
&  i\left(  A_{a}(is+0)-A_{a}(is-0)\right)  ds/2\pi\\
&  =-\frac{1}{\pi}\operatorname{Im}A_{a}(is+0)\,ds=\frac{1}{\psi^{2}%
(a,is)}\dfrac{1}{\pi}\,\operatorname{Im}\frac{-(-2s)ds}{m_{+}(-s^{2}%
+i0,a)+m_{-}(-s^{2}+i0,a)}\\
&  =\psi(a,is)^{-2}d\mu_{a}(-s^{2})=:d\rho_{a}\left(  s\right)
\end{align*}
and (\ref{d rou}) follows. By Corollary \ref{Corol of Faddeev} $\psi
(a,is)^{-2}$ is bound for $a>>1$ and one concludes that the measure
$d\rho\left(  s\right)  $ is finite. It remains to show that $\rho$ is
independent of $a$. To show this we employ the following approximation
arguments. Consider $q_{b}:=\chi_{b}q$ with $b<a$. (e.g. $q_{b}\left(
x\right)  =0,\ x<b\,$\ and $q_{b}\left(  x\right)  =q\left(  x\right)
,\ x\geq b$). Then by Proposition \ref{convergence property} $m_{-}^{b}\left(
\lambda,a\right)  \rightrightarrows$ $m_{-}\left(  \lambda,a\right)  $ in
$\mathbb{C}^{+}$ as $b\rightarrow-\infty$ and hence $\ A_{ab}\rightrightarrows
A_{a}$ in $\mathbb{C}^{+}\diagdown i\Delta$. Therefore, $A_{ab}\rightarrow
A_{a}$ weakly on the boundary of $\mathbb{C}^{+}\diagdown i\Delta$.
Apparently
\[
d\rho_{ab}\left(  s\right)  =%
{\displaystyle\sum_{n=1}^{N_{b}}}
c_{n}^{b}\delta\left(  s-\kappa_{n}^{b}\right)  \ ds,
\]
where $\left\{  -\left(  \kappa_{n}^{b}\right)  ^{2}\right\}  $ are the bound
states of $\mathbb{L}_{q_{b}}$ and $c_{n}^{b}$ are their norming constants,
is, by Corollary \ref{finite rou}, independent of $a$. That is, $\rho
_{ab}=\rho_{b}$. But $\rho\left(  s\right)  =w-\lim\rho_{b}\left(  s\right)  $
as $b\rightarrow-\infty$ which concludes the proof.
\end{proof}

Note that the measure $\rho$ plays the role of 'smeared bound states norming
constants' and can be recovered from $A_{a}~$by (\ref{d rou}). This is the
main value of our split (\ref{split step}). This split has a few alternative
forms. E.g. $R=$ $\left(  A_{a}+\xi_{a}^{-1}\right)  +\left(  r_{a}-1\right)
\xi_{a}{}^{-1}$ also splits $R$ into an analytic function and a small
remainder. Such split (given in a different form) was crucially used in our
\cite{RybCommPDEs2013}. The proof given here appears particularly short.

We find the next consequence of Proposition \ref{step like refl} quite surprising.

\begin{corollary}
If the negative spectrum of $\mathbb{L}_{q}$ is discrete then the sequence
$\left\{  c_{n}\right\}  $ of the right norming constants is summable.
\end{corollary}

The next statement offers some more information on the components in
(\ref{split step}).

\begin{proposition}
[More properties of the analytic split]\label{combin}The function $r_{a}%
\xi_{a}{}^{-1}$ in (\ref{split step}) can further be split into
\begin{equation}
r_{a}\left(  k\right)  =\xi_{a}{}\left(  k\right)  R_{a}\left(  k\right)
-y\left(  k,a\right)  ^{-1}T_{a}\left(  k\right)  , \label{Split 2}%
\end{equation}
where $T_{a},R_{a}$ are the transmission and reflection coefficients for
$q_{a}=\chi_{a}q$. We have
\begin{equation}
R_{a}\left(  k\right)  =O\left(  1/k\right)  ,\ k\rightarrow\pm\infty
;\ R_{a}\left(  k\right)  =o\left(  1/a\right)  ,\ a\rightarrow\infty
,\ \label{R asymp}%
\end{equation}
and $T_{a}$ has at most one pole $\kappa_{a}$ subject to
\begin{equation}
\kappa_{a}=o\left(  1/a\right)  ,\ \ a\rightarrow\infty. \label{kappa a}%
\end{equation}
The function $A_{a}$ in (\ref{split step}) has the property: for $C$ large
enough
\begin{equation}
\left\vert \xi_{a}{}\left(  k\right)  A_{a}\left(  k\right)  \right\vert
\lesssim_{a,q}1\text{ for }\left\vert k\right\vert \geq C,\ \operatorname{Im}%
k\geq0. \label{uniform decay}%
\end{equation}
If $\Delta=\left\{  \kappa_{n}\right\}  \in l^{1}$ then $A_{a}$ has the form
similar to (\ref{R rep 2})
\begin{equation}
A_{a}\left(  k\right)  =\xi_{a}^{-1}\left(  k\right)  S_{a}\left(  k\right)
/B\left(  k\right)  \label{factor}%
\end{equation}
where $S_{a}\in H^{\infty},\ \left\Vert S_{a}\right\Vert _{\infty}\leq
2,~\ $and
\begin{equation}
B\left(  k\right)  =\prod_{n\geq1}\frac{k-i\kappa_{n}}{k+i\kappa_{n}}.
\label{B}%
\end{equation}

\end{proposition}

\begin{proof}
Equation (\ref{Split 2}) is nothing but rearranged (\ref{R1}) written for
$q_{a}$. The asymptotics (\ref{R asymp}) directly follow from Theorem
\ref{th6.1}. The asymptotics (\ref{kappa a}) holds due to the Thirring-Lieb
inequality
\[
\kappa_{a}\lesssim\int_{a}^{\infty}\left\vert q\right\vert <\frac{1}{a}%
\int_{a}^{\infty}x\left\vert q\left(  x\right)  \right\vert dx.
\]
The rest of the statement is a bit harder. It follows from (\ref{A1}) that
\[
\xi_{a}{}A_{a}=y_{a}^{-2}g_{a},
\]
where%
\[
y_{a}\left(  k\right)  :=y\left(  k,a\right)  ,\ \ g_{a}\left(  k\right)
:=\frac{2ik}{m_{+}(k^{2},a)+m_{-}(k^{2},a)}.
\]
By Corollary \ref{Corol of Faddeev}, (\ref{uniform decay}) is then equivalent
to $\left\vert g_{a}\left(  k\right)  \right\vert \lesssim_{a,q}1$ for
$\left\vert k\right\vert \geq C,\ \operatorname{Im}k\geq0.$ The Atkinson
classical result \cite{Atkinson81} says that $m_{\pm}(k^{2},a)=ik+o\left(
1\right)  $ as $\left\vert k\right\vert \rightarrow\infty,\ 0<\varepsilon<\arg
k<\pi-\varepsilon$, which by Corollary \ref{Hergtotz rep 2} implies that
\begin{align*}
g_{a}(\lambda)  &  =2ik\int\frac{d\mu_{a}(s)}{k^{2}-s}\ =g_{a}^{+}\left(
k\right)  +g_{a}^{-}\left(  k\right)  ,\\
g_{a}^{\pm}\left(  k\right)   &  :=2ik\int_{\mathbb{R}_{\pm}}\frac{d\mu
_{a}(s)}{k^{2}-s},
\end{align*}
where $d\mu_{a}$ is given by (\ref{dmu}), and
\begin{equation}
g_{a}(\lambda)=1+o\left(  k^{-1}\right)  ,\ \left\vert k\right\vert
\rightarrow\infty,\ 0<\varepsilon<\arg k<\pi-\varepsilon. \label{f asym}%
\end{equation}
By Proposition \ref{step like refl}, $\operatorname*{Supp}\left(  \mu
_{a}\right)  \cap\mathbb{R}_{-}\subseteq\left[  -h_{0}^{2},0\right]  $ and we
clearly have%
\begin{equation}
g_{a}^{-}\left(  k\right)  =2ik\int_{-h_{0}^{2}}^{0}\frac{d\mu_{a}\left(
s\right)  }{k^{2}-s}=O\left(  \frac{1}{k}\right)  ,\ \left\vert k\right\vert
\rightarrow\infty. \label{G-}%
\end{equation}

Thus we are done if we show%
\begin{equation}
G_{a}\left(  k\right)  :=\frac{k}{k+i}g_{a}^{+}\left(  k\right)  \in
H^{\infty}. \label{G+}%
\end{equation}
To this end consider, as before, $q_{b}=q\chi_{b}$ first. We are now under
conditions of the previous subsection and Proposition \ref{prop7.2} applies.
In particular, (in obvious notation)
\begin{equation}
g_{ab}\left(  k\right)  =\psi\left(  a,k\right)  ^{2}A_{ab}\left(  k\right)
=y_{a}\left(  k\right)  ^{2}S_{ab}\left(  k\right)  /B_{b}\left(  k\right)  ,
\label{g}%
\end{equation}
where $B_{b}$ is the (necessarily finite) Blaschke product with zeros at
$\left\{  i\kappa_{n}^{b}\right\}  $, $\ y_{a}^{2}S_{ab}\in H^{\infty}\cap
C,\ $and uniformly in $b<a$
\[
\left\Vert y_{a}^{2}S_{ab}\right\Vert _{\infty}\leq2\left\Vert y_{a}%
\right\Vert _{\infty}^{2}\lesssim_{q_{a}}1.
\]
One concludes from the representation
\begin{equation}
G_{ab}\left(  k\right)  =\frac{k}{k+i}g_{ab}\left(  k\right)  -\frac{k}%
{k+i}g_{ab}^{-}\left(  k\right)  \label{g+}%
\end{equation}
that $G_{ab}$ is analytic in $\mathbb{C}^{+}$. Moreover, since $g_{ab}^{-}$ is
a rational function with poles at $\left\{  \pm i\kappa_{n}^{b}\right\}  $,
$G_{ab}$ is continuous on $\mathbb{R}$ and it follows from (\ref{g+}),
(\ref{g}), (\ref{G-}), and Proposition \ref{prop7.2} that%
\begin{align*}
\lim G_{ab}\left(  k\right)   &  =\lim g_{ab}\left(  k\right)  -\lim
g_{ab}^{-}\left(  k\right) \\
&  =1,~\left\vert k\right\vert \rightarrow\infty,\ \operatorname{Im}k\geq0.
\end{align*}
Thus, by the Phragm\'{e}n--Lindel\"{o}f principle $G_{ab}\in H^{\infty}\cap C$
for any $b<a$ and on the real line, uniformly in $b$ (for $\left\vert
b\right\vert $ large), we have%
\begin{equation}
\left\Vert G_{ab}\right\Vert _{\infty}\leq4\left\Vert y_{a}\right\Vert
_{\infty}^{2}+\sup_{k\in\mathbb{R}}\left\vert \frac{2k^{2}}{k+i}%
\int_{\mathbb{R}_{-}}\frac{d\mu_{ab}(s)}{k^{2}-s}\right\vert \lesssim_{a,q}1.
\label{bound}%
\end{equation}
Here we have used the fact that due to Proposition \ref{convergence property},
$\mu_{a}\left(  s\right)  =w-\lim_{b\rightarrow-\infty}\mu_{ab}\left(
s\right)  $. By the same proposition, uniformly in $\mathbb{C}^{+}$
\begin{align*}
\ \ G_{ab}\left(  k\right)   &  =\frac{k}{k+i}\left(  \psi\left(  a,k\right)
^{2}A_{ab}\left(  k\right)  -2ik\int_{\mathbb{R}_{-}}\frac{d\mu_{ab}(s)}%
{k^{2}-s}\right) \\
&  \underset{b\rightarrow-\infty}{\rightrightarrows}\frac{k}{k+i}\left(
\psi\left(  a,k\right)  ^{2}A_{a}\left(  k\right)  -2ik\int_{\mathbb{R}_{-}%
}\frac{d\mu_{a}(s)}{k^{2}-s}\right)  \text{ }\\
&  =G_{a}\left(  k\right)
\end{align*}
which combined with (\ref{bound}) proves (\ref{G+}).

It remains to show (\ref{factor}). A bit more complicated approximation of $q$
is required. Split
\[
q=q_{+}-q_{-}\text{ where }q_{\pm}:=\pm\frac{1}{2}\left(  q\pm\left\vert
q\right\vert \right)  \geq0.
\]
and consider $q_{bc}=\chi_{c}q_{+}-\chi_{b}q_{-}$ with $b,c<a$. Clearly
\[
q_{bc}\underset{c\rightarrow-\infty}{\longrightarrow}q_{b}%
\underset{b\rightarrow-\infty}{\longrightarrow}q\text{ in }%
L_{\operatorname*{loc}}^{1}.
\]
Then by Proposition \ref{convergence property} $m_{-}^{bc}\left(
\cdot,a\right)  \underset{c\rightarrow-\infty}{\rightrightarrows}m_{-}%
^{b}\left(  \cdot,a\right)  \underset{b\rightarrow-\infty}{\rightrightarrows
}m_{-}\left(  \cdot,a\right)  $ in $\mathbb{C}^{+}$ and hence
\begin{equation}
A_{abc}\underset{c\rightarrow-\infty}{\rightrightarrows}A_{ab}%
\underset{b\rightarrow-\infty}{\rightrightarrows}A_{a}\text{ in }%
\mathbb{C}^{+}. \label{converg of As}%
\end{equation}
Since $q_{bc}$ is clearly subject to Proposition \ref{prop7.2}, the
representation (\ref{R rep 2}) is valid. We are done then if we show that
\begin{equation}
B_{bc}\underset{c\rightarrow-\infty}{\rightrightarrows}B_{b}%
\underset{b\rightarrow-\infty}{\rightrightarrows}B\text{ in }\mathbb{C}^{+}
\label{converg of Blaschek}%
\end{equation}
as (\ref{converg of As}) and (\ref{converg of Blaschek}) will immediately
imply that the limit $\lim\limits_{b\rightarrow-\infty}$ $\lim
\limits_{c\rightarrow-\infty}S_{abc}$ exists on compacts in $\mathbb{C}^{+}$
and defines an $H^{\infty}$ function $S_{a}$ satisfying $\left\Vert
S_{a}\right\Vert _{\infty}\leq2$.

We make use of a well-known general perturbation principle which in our
particular case, loosely speaking, says that the (negative) bound states of
$\mathbb{L}_{q_{bc}}$ move in unison rightward (leftward) as $c$ ($b$) moves
leftward and new bound states may disappear at (appear from) $0$ only.
Together with Proposition \ref{convergence property} this means that the
(finite) Blaschke product $B_{bc}$ converges to a finite Blaschke product
$B_{b}$\footnote{$B_{b}\left(  k\right)  $ could be $1$.} and, in turn, by
Proposition \ref{Product analog of m-test} $B_{b}$ converges to the (infinite)
Blaschke product $B$ given by (\ref{B}).
\end{proof}

We conclude this section with the following explicitly solvable case which
appears illustrative.

\begin{example}
\label{example} If $q(x)$ is a pure step function, i.e. $q(x)=-h^{2}%
,\;x<0,\;q(x)=0,\;x\geq0$ then $\operatorname*{Spec}\left(  \mathbb{L}%
_{q}\right)  =(-h^{2},\infty)$ and purely a.c., $(-h^{2},0)$ and $\left(
0,\infty\right)  $ being its simple and two fold components respectively.
Moreover%
\[
R(k)=-\left(  \frac{h}{\sqrt{k^{2}}+\sqrt{k^{2}+h^{2}}}\right)  ^{2}%
,\ \ d\rho\left(  s\right)  =\frac{2s}{\pi h^{2}}\sqrt{h^{2}-s^{2}}ds.
\]
The function $y\left(  \cdot,x\right)  ^{-1}\in H^{\infty}$ for any
$x>-\dfrac{\pi}{2h}$.
\end{example}

\section{The IST\ Hankel Operator\label{Our hankel}}

The previous section suggests that the Hankel operator arising in the IST has
a very specific structure. In this section we state and prove some of its
properties of principal importance.

\begin{definition}
[IST Hankel operator]\label{def IST HO} Assume that initial data $q$ is
subject to Hypothesis \ref{hyp1.1}. Let $R$ and $\rho$ be as in Definition
\ref{def R} and Proposition \ref{step like refl} respectively. We call the
Hankel operator
\[
\mathbb{H}(x,t):=\mathbb{H}(\varphi_{x,t}),
\]
with the symbol
\begin{equation}
\varphi_{x,t}(k)=\xi_{x,t}(k)R(k)+\int^{h_{0}}_{0}\frac{\xi_{x,t}%
(is)\,d\rho(s)}{s+ik}, \label{eq9.9}%
\end{equation}
the IST Hankel operator associated with $q$.
\end{definition}

Here is the main result of this section

\begin{theorem}
[Fundamental properties of the IST Hankel operator]\label{th9.5} Under
Hypothesis \ref{hyp1.1} the IST Hankel operator $\mathbb{H}(x,t)$ is
well-defined and has the properties: for any $x\in\mathbb{R},\;t>0$

\begin{enumerate}
\item[(1)] $\mathbb{H}(x,t)$ is selfadjoint,

\item[(2)] $\mathbb{H}(x,t)$ is compact,

\item[(3)] $\mathbb{I}+\mathbb{H}(x,t)>0.$
\end{enumerate}
\end{theorem}

\begin{proof}
Without loss of generality we assume $a=0$. Consider the principal part (see
(\ref{eq4.3})) of $\xi_{x,t}A_{0}$:
\[
\left(  \widetilde{\mathbb{P}}_{-}\xi_{x,t}A_{0}\right)  (k)=-\frac{1}{2\pi
i}\int\left(  \frac{1}{\lambda-(k-i0)}-\frac{1}{\lambda+i)}\right)  \xi
_{x,t}(\lambda)A_{0}(\lambda)d\lambda
\]
Deform the contour of integration to $R+ih,\ h>h_{0}$. It can be easily
justified due to (\ref{uniform decay}) and the rapid decay of $\xi
_{x,t}(\lambda)$ if $\left\vert \lambda\right\vert \rightarrow\infty$ along
$R+ih$ for arbitrary $h>0$.

So, we have
\begin{align}
\left(  \widetilde{\mathbb{P}}_{-}\xi_{x,t}A_{0}\right)  (k)=  &  -\frac
{1}{2\pi i}\int_{\mathbb{R}+ih}\frac{\xi_{x,t}(\lambda)A_{0}(\lambda)}%
{\lambda-k}\ d\lambda-\int_{\mathbb{R}+ih}\frac{\xi_{x,t}(\lambda
)A_{0}(\lambda)}{\lambda+i}\ d\lambda\nonumber\\
&  -\int_{0}^{h_{0}}\frac{\xi_{x,t}(is)}{s+ik}\ d\rho(s)-\frac{1}{i}\int%
_{0}^{h_{0}}\frac{\xi_{x,t}(is)\ d\rho(s)}{s+1}. \label{eq8.2.2}%
\end{align}
It is easy to see that the function
\[
\Phi(k):=-\frac{1}{2\pi i}\int_{\mathbb{R}+ih}\frac{\xi_{x,t}(\lambda
)A_{0}(\lambda)}{\xi-k}\ d\lambda
\]
belong to $C$. Since $\xi_{x,t}A_{0}\in L^{\infty}$ the third term in
(\ref{eq8.2.2})
\begin{equation}
\phi_{x,t}(k):=\int_{0}^{h_{0}}\frac{\xi_{x,t}(is)}{s+ik}\ ds \label{eq8.3.2}%
\end{equation}
belongs to BMO and by Proposition \ref{prop4.4.1} operator $\mathbb{H}(x,t)$
is well-defined and bounded. Statement (1) follows then from Propositions
\ref{prop4.7} and \ref{props of R}.

To prove (2)\ we observe (see (\ref{eq9.9}), Theorem \ref{th4.5} and
(\ref{eq4.5})) that
\[
\mathbb{H}(x,t)=\mathbb{H}\left(  \Phi+r_{0}\xi_{0}^{-1}\right)  .
\]
Since $\Phi\in C$ and $r_{0}\in C$ (due to (\ref{r_a})) we see that Theorem
\ref{thHart} implies statement (2).

We are ready now to prove part (3). Let $f\in H^{2}$, then
\begin{equation}
\langle\left(  I+\mathbb{H}(x,t)\right)  f,f\rangle=\langle f,f\rangle
+\langle\mathbb{H}(\xi_{x,t}R)f,f\rangle+\langle\mathbb{H}(\phi_{x,t}%
)f,f\rangle, \label{eq8.4.2}%
\end{equation}
where $\phi_{x,t}$ is given by (\ref{eq8.3.2}). Since $\phi_{x,t}\in$ BMO the
last term of (\ref{eq8.4.2}) exists and
\begin{align}
\langle\mathbb{H}(\phi_{x,t})f,f\rangle &  =-i\int_{0}^{h_{0}}d\rho(s)\langle
J\mathbb{P}_{-}\frac{f}{\cdot-is},f\rangle\nonumber\\
&  =i\int_{0}^{h_{0}}d\rho(s)f(is)\langle\frac{1}{\cdot+is},f\rangle=2\pi
\int_{0}^{h_{0}}\left\vert f(is)\right\vert ^{2}d\rho(s)\geq0. \label{eq8.5.2}%
\end{align}
Since $\Vert\xi_{x,t}R\Vert_{\infty}\leq1$ we have
\begin{equation}
\langle\left(  I+\mathbb{H}(\xi_{x,t}R)\right)  f,f\rangle\geq0.
\label{eq8.6.2}%
\end{equation}
Suppose that (3) does not hold. Then (\ref{eq8.5.2}) and (\ref{eq8.6.2})
imply
\begin{equation}
\int_{0}^{h_{0}}|f(is)|^{2}d\rho(s)=0. \label{eq8.7.2}%
\end{equation}

We need to consider three cases.

\textbf{Case 1}. The support of $\rho$ is a uniqueness set for $H^{2}$. Then
(\ref{eq8.7.2}) implies that $f\equiv0$ and statement (3) trivially follows.

Let now $\operatorname*{Supp}\rho$ be a non uniqueness set for $H^{2}$. Then
$\operatorname*{Supp}\rho=\{i\kappa_{n}\}$ with $\kappa_{n}>0$ subject to the
Blaschke condition and condition (\ref{eq8.7.2}) holds iff $f(i\kappa_{n})=0.$
It follows from the canonical factorization theorem that $f=BF$, where $B$ is
the Blaschke product with zeros $\{i\kappa_{n}\}$ and $F\in H^{2}$. Thus
(\ref{eq8.4.2}) reads
\begin{equation}
\langle\left(  I+\mathbb{H}(x,t)\right)  f,f\rangle=\langle F,F\rangle
+\langle\mathbb{H}(\xi_{x,t}BR)F,BF\rangle. \label{eq8.8.2}%
\end{equation}
By Proposition \ref{combin} (see (\ref{factor})) $BA_{0}\in H^{\infty}$.
Moreover, by Proposition \ref{step like refl} $r_{0}\ \xi_{0}^{-1}\in C$ and
by Theorem \ref{thGru01} $\xi_{x,t}\in H^{\infty}+C$. Thus we have $\xi
_{x,t}BR\in H^{\infty}+C$.

\textbf{Case 2.} $\operatorname*{Supp}\rho$ is not a uniqueness set and $R$ is
not unimodular function. By Theorem \ref{rem5.8} we have
\begin{equation}
\Vert\mathbb{H}(\xi_{x,t}BR)\Vert<1 \label{eq8.9.2}%
\end{equation}
and (\ref{eq8.9.2}) implies
\[
|\langle\mathbb{H}(\xi_{x,t}BR)F,BF\rangle|\leq\Vert\mathbb{H}(\xi
_{x,t}BR)\Vert\cdot\Vert F\Vert_{2}\Vert BF\Vert_{2}<\Vert F\Vert_{2}^{2},
\]
which immediately yields statement (3).

\textbf{Case 3.} $\operatorname*{Supp}\rho$ is not a uniqueness set and $R$ is
a unimodular function. Then by Theorem \ref{thGru01} (see (\ref{eq5.3}))
$\xi_{x,t}=B_{x,t}u_{x,t}$ with some infinite Blaschke product $B_{x,t}$ and
unimodular function $u_{x,t}$ from $C$. Therefore $\xi_{x,t}BR\in H^{\infty
}+C$ and by Lemma \ref{lem5.6} $\xi_{x,t}BR$ is not invertible in $H^{\infty
}+C$. By Theorem \ref{th5.7} then (\ref{eq8.9.2}) holds and as in Case 2
statement (3) follows.
\end{proof}

\begin{remark}
\label{rem9.6}Theorem \ref{th9.5} says that $(\mathbb{I}+\mathbb{H}%
(x,t))^{-1}$ is a bounded operator for any $x\in\mathbb{R}$ and $t>0$, which
is of course of a particular importance for validation of the IST. Cases 1,2
in the proof are easy and were done in \cite{RybNON2011}. Case 3 is much more
subtle. Under assumption that $h_{0}=0$ in Hypothesis \ref{hyp1.1} and
$q_{a}=0$ it was proven in our \cite{GruRyPAMS13}. Then in \cite{GruRyOTAA13}
we relaxed the condition $h_{0}=0$ but imposed some extra conditions on the
negative spectrum of $\mathbb{L}_{q}$. In the full generality Theorem
\ref{th9.5} appears first in this paper and is one of our main results.
\end{remark}

\begin{remark}
Theorem \ref{th9.5} does not say that if we split $\mathbb{H}\left(
x,t\right)  $ into two Hankel operators corresponding to the two pieces in
(\ref{eq9.9}) then each Hankel operator is compact\footnote{Some subtle
conditions for $\mathbb{H}(\phi_{x,t})\in\mathfrak{S}_{\infty}$ are studied in
our \cite{GruRyOTAA13}.}. However, if we notice that $\mathbb{H}(\phi_{x,t})$
is unitary equivalent to the integral operator (\ref{eq4.10}) on
$L^{2}(\mathbb{R}_{+})$ with the continuous kernel $h\left(  \cdot\right)
=\int_{0}^{h_{0}}e^{-s(\cdot)}\xi_{x,t}\left(  is\right)  d\rho(s)$ then%
\[
\operatorname*{tr}\mathbb{H}(\phi_{x,t})=\int_{0}^{\infty}dz\int_{0}^{h_{0}%
}e^{-2zs}\xi_{x,t}\left(  is\right)  d\rho(s)=\frac{1}{2}\int_{0}^{h_{0}}%
\xi_{x,t}\left(  is\right)  \frac{d\rho(s)}{s},
\]
which means that $\mathbb{H}(\phi_{x,t})\in\mathfrak{S}_{1}$ iff $\int%
_{0}^{h_{0}}d\rho(s)/s$ is bounded. It is clearly the case in Example
\ref{example} but untrue in general. Exploiting the same unitary equivalence
argument, one can easily prove that $\mathbb{H}(\phi_{x,t})$ is bounded iff
$\rho$ is a Carleson measure, i.e.
\begin{equation}
\sup\left\{  \frac{1}{\delta}\int_{0}^{\delta}d\rho:\delta>0\right\}  <\infty.
\label{eq9.2}%
\end{equation}
We find this result quite interesting as conditions like (\ref{eq9.2}) are
frequently a priori assumed even in the case when $q$ tends to a constant at
$-\infty$ (c.f. \cite{Venak86}). Also note that (\ref{eq9.2}) means that $0$
must not be an eigenvalue of $\mathbb{L}_{q}$. Of course, it follows from
(\ref{G+}) that there are no positive (imbedded) bound states either. Thus,
Hypothesis \ref{hyp1.1} imposes a restriction on the spectrum: the discrete
spectrum of $\mathbb{L}_{q}$ could only be negative.
\end{remark}

\section{Singular numbers of the IST Hankel operator\label{s-numbers}}

As well-known a bounded operator $\mathbb{A}$ is compact ($\mathbb{A}%
\in\mathfrak{S}_{\infty}$) if it can be uniformly approximated by rank $n$
operators $\mathbb{A}_{n}$. Singular numbers $s_{n}\left(  \mathbb{A}\right)
$ give an accurate quantitative description of the rate of convergence
$\mathbb{A}_{n}\rightarrow\mathbb{A}$. In the context of Hankel operators
singular numbers gain a whole new meaning as $s_{n}\left(  \mathbb{H}\left(
\varphi\right)  \right)  $ are directly related to best rational
approximations of $\varphi$. Consequently, since about 1970, sparked by
seminal works due to Adamyan-Arov-Krein, a large variety of issues related to
singular numbers\footnote{In particular, membership in $\mathfrak{S}_{p}$
classes.} have been extensively studied (see, e.g. \cite{Nik2002},
\cite{Peller2003} and the extensive literature cited therein). We however are
not sure if any of these has been used in soliton theory. In this section we
shall demonstrate how the Adamyan-Arov-Krein classical theory beautifully
yields subtle relations between the decay of $s_{n}\left(  \mathbb{H}\left(
x,t\right)  \right)  $ and properties of the initial data $q$. In the
subsequent section this will be translated into substantial conclusions on the
initial value problem for the KdV equation. Due to space limitations, we focus
on the opportunities that the theory of Hankel operators promises rather than
completeness of our results.

\subsection{Some general statements on singular numbers of Hankel operators}

Let $\mathcal{R}_{n\text{ }}$denote the set of rational functions bounded at
infinity with all poles in the upper half plane of total multiplicity $\leq
n$. The following theorems are fundamental in the study of singular numbers of
Hankel operators.

\begin{theorem}
[Adamyan-Arov-Krein, 1971]\label{AAK} Let $\varphi\in L^{\infty}$. Then
\[
s_{n}\left(  \mathbb{H}\left(  \varphi\right)  \right)  =\operatorname*{dist}%
\nolimits_{L^{\infty}}\left(  \varphi,\mathcal{R}_{n}+H^{\infty}\right)  .
\]

\end{theorem}

\begin{theorem}
[Bernstein-Jackson, 1910]\label{Jackson}Let $\varphi\in C^{m}$. Then
\[
\operatorname*{dist}\nolimits_{L^{\infty}}\left(  \varphi,\mathcal{R}%
_{n}+H^{\infty}\right)  \lesssim\left\Vert \varphi^{\left(  m\right)
}\right\Vert _{\infty}/n^{m}.
\]

\end{theorem}

We will need a simple

\begin{lemma}
\label{Tech lemma}Let $C,b>0,$ $p\geq1$ and $\left\{  s_{n}\right\}  _{n\geq
1}$ be a positive sequence. If
\begin{equation}
s_{n}\leq\frac{C\left(  m!\right)  ^{p}}{b^{m}}\frac{1}{n^{m}},\ \ \ \forall
m\in\mathbb{N}_{0} \label{ined cond 1}%
\end{equation}
then
\begin{equation}
s_{n}\leq C2^{p}\exp\left\{  -\left(  p/2\right)  \left(  bn\right)
^{1/p}\right\}  . \label{tech ineq 1}%
\end{equation}

\end{lemma}

\begin{proof}
Without loss of generality we may set $C=1=b$. Multiplying $\frac{n^{m}%
}{\left(  m!\right)  ^{p}}s_{n}\leq1$ by $2^{-m}$ and then summing on $m\geq0$
yields%
\begin{equation}
s_{n}%
{\displaystyle\sum_{m\geq0}}
\dfrac{n^{m}}{\left(  m!\right)  ^{p}}\frac{1}{2^{m}}\leq%
{\displaystyle\sum_{m\geq0}}
\frac{1}{2^{m}}=2. \label{long ineq}%
\end{equation}
Bound now the left hand side of (\ref{long ineq}) from below by the Jensen
inequality:%
\begin{align}
s_{n}%
{\displaystyle\sum_{m\geq0}}
\dfrac{n^{m}}{\left(  m!\right)  ^{p}}\frac{1}{2^{m}}  &  \geq s_{n}%
2^{1-p}\left(
{\displaystyle\sum_{m\geq0}}
\frac{\left(  n^{1/p}/2\right)  ^{m}}{m!}\right)  ^{p}\label{ineq}\\
&  =s_{n}2^{1-p}\left(  \exp\left(  n^{1/p}/2\right)  \right)  ^{p}.\nonumber
\end{align}
Combining (\ref{tech ineq 1}) and (\ref{ineq}) proves the lemma.
\end{proof}

Theorems \ref{AAK}, \ref{Jackson} and Lemma \ref{Tech lemma} immediately yield
the following observation.

\begin{proposition}
\label{prop on sn}Let $f\in L^{1}$, $h>0$ and
\begin{equation}
\varphi(k)=\int\frac{f(s)}{s+ih-k}ds. \label{small fi}%
\end{equation}
Then
\[
s_{n}\left(  \mathbb{H}\left(  \varphi\right)  \right)  \leq\left(
2/h\right)  \left\Vert f\right\Vert _{1}\exp\left\{  -\left(  h/2\right)
n\right\}  .
\]

\end{proposition}

We conclude this subsection with a useful

\begin{remark}
\label{Remark Ky fan}Due to the well-known Ky Fan inequality%
\begin{equation}
s_{n+m-1}\left(  \mathbb{A}+\mathbb{B}\right)  \leq s_{n}\left(
\mathbb{A}\right)  +s_{m}\left(  \mathbb{B}\right)  , \label{Ky Fan}%
\end{equation}
a finite rank perturbation has no effect on asymptotics of singular numbers.
By Theorem \ref{AAK} $\mathbb{H}\left(  \varphi\right)  $ is finite rank iff
$\varphi$ is rational. This means that the rational part of a symbol $\varphi$
does not influence the asymptotics of $s_{n}\left(  \mathbb{H}\left(
\varphi\right)  \right)  \ \ n\rightarrow\infty$.
\end{remark}

\subsection{Rate of decay of singular numbers of the IST\ Hankel operator
\label{rate of decay}}

With the preparatory material out of the way, we turn now to the actual
results of this section. The following easily verifiable formula will be used
($t>0$)
\begin{equation}
\left\vert \xi_{x+iy,t}(\alpha+i\beta)\right\vert =\xi_{x,t}(i\beta
)\exp\left\{  -\left(  \sqrt{24\beta t}\alpha+\frac{y}{\sqrt{24\beta t}%
}\right)  ^{2}+\frac{y^{2}}{24\beta t}\right\}  . \label{ksi}%
\end{equation}

\begin{theorem}
[Asymptotics of singular numbers]\label{Thm on s-numbers}Assume Hypothesis
\ref{hyp1.1} with $w\left(  x\right)  =\exp\left(  \gamma x^{1/\delta}\right)
$ where $\gamma,\delta>0$. Then there exists a constant $C$ dependent on
$\gamma,\delta$ such that uniformly on compacts\footnote{Much more specific
statements regarding domains can be made.} of $\left(  x,t\right)  $
\begin{equation}
s_{n}\left(  \mathbb{H}\left(  x,t\right)  \right)  =O\left(  \exp\left\{
-Cn^{\omega}\right\}  \right)  ,\ \ n\rightarrow\infty, \label{singular}%
\end{equation}
where $\omega=1$ if $0<\delta\leq1$ and $\omega=1/\delta$ if $\delta>1$.
\end{theorem}

\begin{proof}
As before whenever it leads to no confusion, we suppress the dependence on
$\left(  x,t\right)  $ and assume $a=0$. By taking the co-analytic part of
$\xi R$ and splitting $R$ by (\ref{Split 2}) one has%
\begin{equation}
\mathbb{H}\left(  x,t\right)  =\mathbb{H}(\Psi-\frac{c_{0}\xi\left(
i\kappa_{0}\right)  }{\kappa_{0}+i\cdot}+\xi R_{0}), \label{h split}%
\end{equation}
where $c_{0}$ is the norming constant for the bound state\footnote{If there is
no bound state then $c_{0}=0$.} $-\kappa_{0}^{2}$ and ($y_{0}:=y\left(
\cdot,0\right)  $)%
\begin{equation}
\Psi(k):=-\frac{1}{2\pi i}\int_{\mathbb{R}+ih}\left(  \cdot-k\right)  ^{-1}%
\xi\left(  A_{0}-y_{0}^{-1}T_{0}\right)  . \label{big fi}%
\end{equation}
For $R_{0}$, the reflection coefficient from $q_{0}$, we use the
representation \cite{Deift79}%
\begin{equation}
R_{0}\left(  \lambda\right)  =\frac{T_{0}\left(  \lambda\right)  }{2i\lambda
}\int_{0}^{\infty}e^{-2i\lambda s}g\left(  s\right)  ds, \label{R_0}%
\end{equation}
where $g$ is some function for which we only need the bound
\[
\left\vert g\left(  s\right)  \right\vert \leq\left\vert q\left(  s\right)
\right\vert +const\int_{s}^{\infty}\left\vert q\right\vert .
\]
If $0<\delta\leq1$ then the Fourier transform $G\left(  \lambda\right)
:=\int_{0}^{\infty}e^{-2i\lambda s}g\left(  s\right)  ds$ in (\ref{R_0})
extends analytically to (at least) a strip and can be treated similarly to
(\ref{big fi}). In the case $\delta>1$ the part $\xi R_{0}$ of the symbol in
(\ref{h split}) need not extend analytically from the real line. But it has a
pseudo analytic extension. Due to our condition on $q$, the function $G\left(
\lambda\right)  $ admits (see, e.g. \cite{RybCommPDEs2013})\ a smooth bounded
pseudo analytic continuation $G\left(  \alpha,\beta\right)  $\ into
$\mathbb{C}^{+}$ with the property ($\delta>1$)%
\begin{equation}
\left\vert \overline{\partial}G\left(  \alpha,\beta\right)  \right\vert
\lesssim K\exp\left\{  -\left(  Q/\beta\right)  ^{1/\left(  \delta-1\right)
}\right\}  ,\ \ \ K:=\int_{0}^{\infty}w\left\vert q\right\vert , \label{dG}%
\end{equation}
where $Q$ is a constant dependent on $\gamma,\delta$. We are able now to
evaluate the co-analytic part of $\xi R_{0}$ by the Green formula applied to
the domain $\left\{  0<\beta<h\right\}  $ with any $h>\kappa_{0}$%
\begin{align}
\widetilde{\mathbb{P}}_{-}(\xi R_{0})  &  =\frac{1}{2\pi i}\int\frac
{\xi\left(  \lambda\right)  R_{0}\left(  \lambda\right)  }{k-i0-\lambda
}d\lambda\label{Rxt}\\
&  =\frac{1}{\pi}\int_{0<\beta<h}\frac{T_{0}\left(  \lambda\right)
}{2i\lambda}\frac{\xi\left(  \lambda\right)  \overline{\partial}G\left(
\alpha,\beta\right)  }{\lambda-k}d\alpha d\beta\label{Rxt 2}\\
&  +\frac{1}{2\pi i}\int_{\mathbb{R}+ih}\frac{T_{0}\left(  \lambda\right)
}{2i\lambda}\frac{\xi\left(  \lambda\right)  G\left(  \alpha,\beta\right)
}{k-\lambda}d\lambda\label{Rxt 3}\\
&  +\frac{\omega}{\kappa_{0}+ik}, \label{Rxt 4}%
\end{align}
where $\omega$ is an essential constant. We now insert (\ref{Rxt}) into
(\ref{h split}). The term (\ref{Rxt 3}) can be combined with $\Psi$ to form a
new one, $\Omega$. The term (\ref{Rxt 4}) joints in (\ref{h split}) the middle
one to produce a partial fraction $\frac{\sigma}{\kappa_{0}+ik}$ which by
Remark \ref{Remark Ky fan} can be neglected. Due to the rapid decay of
$\xi\left(  \lambda\right)  $ along $\mathbb{R}+ih$, Proposition
\ref{prop on sn} applies and choosing $h=2h_{0}$ we get
\begin{equation}
s_{n}\left(  \mathbb{H}(\Psi+\frac{\sigma}{\kappa_{0}+i\cdot})\right)
=O\left(  e^{-h_{0}n}\right)  . \label{1}%
\end{equation}

The term (\ref{Rxt 2}), which we denote by $\Lambda$, is therefore only one
that needs some attention. By (\ref{ksi}) and (\ref{dG}) we have%
\[
\left\vert \xi_{x+iy,t}\left(  \lambda\right)  \overline{\partial}G\left(
\alpha,\beta\right)  \right\vert \lesssim K\xi_{x,t}(i\beta)\exp\left\{
-\left(  \sqrt{24\beta t}\alpha\right)  ^{2}-\left(  Q/\beta\right)
^{1/\left(  \delta-1\right)  }\right\}  .
\]

Omitting straightforward but rather involved computations, we obtain%
\[
\left\Vert \partial_{k}^{m}\Lambda\right\Vert _{\infty}\lesssim\left(
m!\right)  ^{\delta}\widetilde{Q}^{-m}%
\]
with some $\widetilde{Q}$ dependent on $\gamma$ and $\delta$. Theorems
\ref{AAK}, \ref{Jackson} and Lemma \ref{Tech lemma} then yield
\begin{equation}
s_{n}\left(  \mathbb{H}(\Lambda)\right)  =O\left(  \exp\left\{  -Cn^{1/\delta
}\right\}  \right)  . \label{2}%
\end{equation}
Combining (\ref{1}) and (\ref{2}) through (\ref{Ky Fan}) implies
(\ref{singular}) with some smaller than in (\ref{2}) constant $C.$
\end{proof}

Finally, the following theorem can be obtained by using techniques from this
section (see \cite{RybCommPDEs2013} for more detail)$.$

\begin{theorem}
\label{anal}Under conditions of Theorem \ref{Thm on s-numbers} for any $t>0$
the operator-valued function $\partial_{t}^{m}\partial_{x}^{m}\mathbb{H}%
\left(  x,t\right)  $ (as an element of $\mathfrak{S}_{p}$, $0<p\leq\infty$)
is (1) entire if $0<\delta<2$, (2) analytic in the strip $\left\vert
\operatorname{Im}x\right\vert <\frac{9\sqrt{2}}{8}\gamma\sqrt{t}$ if
$\delta=2$, and (3) in the Gevrey class $G^{\delta/2}$ if $\delta>2$.
\end{theorem}

\section{The IST\ for the KdV equation with step-like initial
data\label{main section}}

In this section we finally state and prove our main result, which loosely
speaking says:\ The problem (\ref{KdV})-(\ref{KdVID}) is well posed and its
solution can be found by a suitable IST. With all the preparations done in the
previous sections, the actual proof will be quite short.

Note that while the interest to well-posedness of integrable systems has been
generated by the progress in soliton theory, well-posedness issues are
typically approached by means of PDEs techniques \cite{Tao06} (norm estimates,
etc.) and the IST is not usually employed. In soliton theory, in turn,
well-posedness is commonly assumed (frequently even by default) and one
applies the IST method to study the unique solution to (\ref{KdV}%
)-(\ref{KdVID}) or any other integrable system. The paper \cite{KapTop06}
represents a rather rare example where the complete integrability of
(\ref{KdV})-(\ref{KdVID}) with periodic initial data was used in a crucial way
to prove some subtle well-posedness results for irregular $q$ which are not
accessible by harmonic analysis means. In our case neither a priori
well-posedness nor IST are readily available and we have to deal with both at
the same time.

Solutions of the KdV can be understood in a number of different ways
\cite{Tao06} (classical, strong, weak, etc.) resulting in a variety of
different well-posedness results. Our definition is consistent with that in
\cite{KapTop06}.

\begin{definition}
[Natural solution]\label{WP} Let $\left\{  q_{n}\left(  x,t\right)  \right\}
$ be a sequence of (classical) solutions of (\ref{KdV}) with compactly
supported initial data $q_{n}\left(  x\right)  $ converging in
$L_{\operatorname*{loc}}^{1}$ to $q\left(  x\right)  $. We call $q\left(
x,t\right)  $ a global natural solution to (\ref{KdV})-(\ref{KdVID}) if

\begin{enumerate}
\item for any $t>0$ uniformly on compacts of $\mathbb{R}$
\[
q\left(  x,t\right)  =\lim q_{n}\left(  x,t\right)  ,\ n\rightarrow\infty,
\]
independently of the choice of $q_{n}$.

\item $q\left(  x,t\right)  $ is a classical solution of (\ref{KdV}),

\item $q\left(  x,t\right)  $ satisfies the initial condition (\ref{KdVID}) in
the sense that
\begin{equation}
q\left(  x,t\right)  \rightarrow q\left(  x\right)  \ \text{in~}%
L_{\operatorname*{loc}}^{1}\text{ as }t\rightarrow+0. \label{IC}%
\end{equation}

\end{enumerate}
\end{definition}

Thus we understand well-posedness in a very strong sense. It also looks quite
natural from the computational and physical point of view. Another feature of
Definition \ref{WP} is that existence implies uniqueness and certain
continuous dependence on the initial data.

\begin{theorem}
[Main Theorem]\label{MainThm} Assume that the initial data $q$ in
(\ref{KdVID}) is subject to Hypothesis \ref{hyp1.1} with $w\left(  x\right)
=\exp\left(  \gamma x^{1/\delta}\right)  $ where $\gamma,\delta>0$. Then the
Cauchy problem (\ref{KdV})-(\ref{KdVID}) has a smooth global natural solution
$q(x,t)$ (Definition \ref{WP}) given by
\begin{equation}
q(x,t)=-2\partial_{x}^{2}\log\det\left(  1+\mathbb{H}(x,t)\right)  ,
\label{det_form}%
\end{equation}
where $\mathbb{H}(x,t)$ is the IST Hankel operator associated with $q$
(Definition \ref{def IST HO}). Singular numbers of $\mathbb{H}(x,t)$ decay
uniformly on compacts of $\left(  x,t\right)  $ at the rate%
\begin{equation}
s_{n}\left(  \mathbb{H}\left(  x,t\right)  \right)  =O\left(  \exp\left\{
-Cn^{\omega}\right\}  \right)  ,\ \ n\rightarrow\infty, \label{s_n decay}%
\end{equation}
where $\omega=\min\left\{  1,1/\delta\right\}  $ and $C$ is a constant
dependent on $\gamma,\delta$. Furthermore, for any $t>0$

\begin{enumerate}
\item If $1<\delta<2$ then $q(x,t)$ is meromorphic on $\mathbb{C}$ with no
poles on $\mathbb{R}$.

\item If $\delta=2$ then $q(x,t)$ is meromorphic in the strip
\begin{equation}
\left\vert \operatorname{Im}x\right\vert <\frac{9\sqrt{2}}{8}\gamma\sqrt{t}
\label{strip}%
\end{equation}
with no poles on $\mathbb{R}$.

\item If $\delta>2$ then $q(x,t)$ is in the Gevrey class $G^{\delta/2}$.
\end{enumerate}
\end{theorem}

\begin{proof}
Without loss of generality we can set the splitting point $a=0$. Since the
problem (\ref{KdV})-(\ref{KdVID}) with initial data $\chi q$ is classical, one
only needs to consider $q_{n}$ in Definition \ref{WP} such that $\chi
q_{n}=\chi q$. The KdV equation with initial data $q_{n}\left(  x\right)  $
has a unique classical solution $q_{n}\left(  x,t\right)  $ computed by%
\[
q_{n}\left(  x,t\right)  =-2\partial_{x}^{2}\log\det\left(  \mathbb{I}%
+\mathbb{H}_{n}(x,t)\right)  ,
\]
where $\mathbb{H}_{n}(x,t)$ is the IST Hankel operator corresponding to
$q_{n}$. By Theorem \ref{anal}, $q_{n}\left(  x,t\right)  $ is a meromorphic
function in $x$ on the entire complex plane. Consider the function $q\left(
x,t\right)  $\ given by (\ref{det_form}). By Theorem \ref{anal}, it is well
defined and at least Gevrey smooth. It remains to prove that $q\left(
x,t\right)  =\lim q_{n}\left(  x,t\right)  ,\ n\rightarrow\infty$, solves
(\ref{KdV})-(\ref{KdVID}). By Theorems \ref{th9.5} and \ref{anal} $q\left(
x,t\right)  $ is well defined and at least Gevrey smooth. We rewrite
$q=q_{n}+\Delta q_{n}$, and insert this into (\ref{KdV}):%
\begin{align}
&  \partial_{t}q-6q\partial_{x}q+\partial_{x}^{3}q\label{rhs}\\
&  =\partial_{t}\Delta q_{n}+3\partial_{x}\left[  \left(  \Delta
q_{n}-2q\right)  \Delta q_{n}\right]  +\partial_{x}^{3}\Delta q_{n}.\nonumber
\end{align}
For $\Delta q_{n}$ we have (dropping subscript $x,t$)%
\[
\Delta q_{n}=-2\partial_{x}^{2}\log\det\left(  \mathbb{I}-\left(
\mathbb{I}+\mathbb{H}\right)  ^{-1}\left(  \mathbb{H-H}_{n}\right)  \right)
.
\]
It follows form (\ref{h split}) and (\ref{big fi}) that for the symbol
$\ \Delta\Psi_{n}$ of $\mathbb{H-H}_{n}$ we have
\begin{equation}
\Delta\Psi_{n}(k)=\frac{1}{2\pi i}\int_{\mathbb{R}+ih}\ \ \ \frac{\xi\left(
\lambda\right)  \ \left(  A_{0,n}\left(  \lambda\right)  -A_{0}\left(
\lambda\right)  \right)  d\lambda}{\lambda-k}. \label{delta}%
\end{equation}
But, as we know, $A_{0,n}$ $\rightrightarrows A_{0}$ in $\mathbb{C}^{+}$ as
$n\rightarrow\infty$ and we can easily conclude that $\partial_{t}^{m}%
\partial_{x}^{l}\left(  \mathbb{H-H}_{n}\right)  $ vanishes in the trace norm
as $n\rightarrow\infty$. Therefore $\partial_{t}^{m}\partial_{x}^{l}\Delta
q_{n}\rightarrow0,n\rightarrow\infty$ and (\ref{rhs}) immediately implies that
$q\left(  x,t\right)  $ solves (\ref{KdV}). (\ref{IC}) is proven in
\cite{RybNON2011}. The rest of the statement follows from Theorems
\ref{Thm on s-numbers} and \ref{anal}.
\end{proof}

Theorem \ref{MainThm} concludes our study started in \cite{Ryb10} of step like
initial data with arbitrary behavior at $-\infty$. The relevance to Hankel
operators was realized in \cite{RybNON2011} but we had to impose an additional
condition to prove non-singularity of $\mathbb{I}+\mathbb{H}$. We conjectured
in \cite{RybNON2011} that this condition can be removed. In \cite{GruRyPAMS13}%
, \cite{GruRyOTAA13} we finally linked the IST and Hankel operators but were
unable to completely remove that condition. The new condition was so weak and
subtle that a counterexample would be extremely hard to construct. In
\cite{RybCommPDEs2013} we conveniently used the language of Hankel operators
to prove (1)-(3). The exact connection (\ref{s_n decay}), based upon the
Adamyan-Arov-Krein theory, between the decay of $q\left(  x\right)  $ at
$+\infty$ and smoothness of $\mathbb{H}\left(  x,t\right)  $ (and hence
$q\left(  x,t\right)  $) is first established here. Besides, in our previous
papers we were more dependent on auxiliary results from other sources which
come with somewhat stronger local conditions on $q\left(  x\right)  $. We also
had a more complicated formula for the measure $\rho$ in the scattering data.
The expression (\ref{d rou}) is easiest possible. All this has resulted in a
much more streamlined exposition.

\section{Some corollaries of the main theorem \label{corollaries}}

Theorem \ref{MainThm} readily implies a number of corollaries as well as
quickly recovers and improves on many already known results. Below are some of them.

\subsection{Hirota tau function}

The explicit formula (\ref{det_form}) immediately yields the representation
\[
\tau\left(  x,t\right)  =\det\left(  \mathbb{I}+\mathbb{H}(x,t)\right)  ,
\]
for the Hirota tau function \cite{Hirota71}, a well-known popular object of
soliton theory. The substitution $q\left(  x,t\right)  =-2\partial_{x}^{2}%
\tau\left(  x,t\right)  $ is commonly used as an ansatz to reduce the KdV
equation to the so-called bilinear KdV which is advantageous in some
situations. Formulas like (\ref{det_form}) are particularly convenient for
describing classes of exact solutions (see, e.g. \cite{Ma05}) and $\tau\left(
x,t\right)  $ typically appears as a Wronskian. We refer to \cite{ErcMcKean90}%
, \cite{NovikovetalBook}, \cite{Popper84}, and \cite{Venak86} for
(\ref{det_form}) in the context of the Cauchy problem for the KdV. In our
generality (\ref{det_form}) is new.

\subsection{Rate of convergence}

The relation (\ref{s_n decay}) means that the determinant in (\ref{det_form})
rapidly converges. This fact, coupled with the recent progress in computing
Fredholm determinants \cite{Bornemann2010}, suggests that, contrary to the
common belief, (\ref{det_form}) could be used for numerical evaluations.

\subsection{Analyticity}

Parts (1)-(3) of Theorem \ref{MainThm} say that any, no matter how rough,
locally integrable initial profile $q\left(  x\right)  $ instantaneously
evolves under the KdV flow into a smooth function $q\left(  x,t\right)  $.
This effect, also called dispersive smoothing, has a long history. While being
noticed long ago, its rigorous proof took quit a bit of effort even for box
shaped initial data \cite{Murray(Cohen)78} (see also \cite{Zhou1997} for other
integrable systems). Theorem \ref{MainThm} also implies that the rate of decay
of $q\left(  x\right)  $ at $+\infty$ solely determines smoothness of
$q\left(  x,t\right)  $ for any $q\left(  x\right)  $ essentially bounded from below.

More can be said if $0<\delta<2$. The solution $q\left(  x,t\right)  $ is then
meromorphic in $x$ on the whole complex plane which means that $q\left(
x,t\right)  $ cannot vanish on a set of positive Lebesgue measure for any $t>0
$ unless $q\left(  x\right)  $ is identically zero and we quickly recover and
improve on many results of \cite{Zhang92}. E.g. assuming that $q\left(
x\right)  $ is absolutely continuous and short range, it is proven in
\cite{Zhang92} that $q\left(  x,t\right)  $ cannot have compact support at two
different moments unless it vanishes identically. The techniques of
\cite{Zhang92} also rely on the IST\ and some Hardy space arguments.

If $\delta=2$ then part (2) of Theorem \ref{MainThm} says that $q\left(
x,t\right)  $ is meromorphic in a strip widening proportionally to $\sqrt{t}$.
The closest known result \cite{Tarama04} can only claim that $q\left(
x,t\right)  $ is real analytic and its proof requires strong decay at
$-\infty$ as well as local $L^{2}$ integrability. The approach of
\cite{Tarama04} is based on the classical IST coupled with analysis of the
Airy function and therefore quite involved.

\section{What we don't know but would like to\label{last sect}}

\subsection{Slower decay at $+\infty$}

The most important problem we are particularly concerned with is how much the
decay condition at $+\infty$ could be relaxed. Due to the famous Bourgain
result \cite{Bourgain93} the problem (\ref{KdV})-(\ref{KdVID}) is well-posed
for $q\in L^{2}$. It is important to notice that the Bourgain's paper drew an
enormous attention in the PDE/Harmonic analysis community to the
well-posedness of (\ref{KdV})-(\ref{KdVID}) with singular data from Sobolev
spaces with negative indices (see, e.g., \cite{Tao06} and the extensive
literature cited therein). Thus developing IST techniques\footnote{Or at least
understanding in what sense (\ref{KdV})-(\ref{KdVID}) is more integrable than
a generic PDE.} for such initial data is arguably even more important and is a
long overdue problem. It literally remains an uncharted territory. The formula
(\ref{det_form}) however cannot possibly hold as is for a number of reasons.
We cautiously conjecture that a suitable IST can still be found if $q$ in
Hypothesis \ref{hyp1.1} is merely square integrable at $+\infty$. At this
point we are far from understanding\ how to deal with the mounting serious
issues. In some particular (but interesting) cases the formulas for $R$ and
$\rho$ in (\ref{eq8.2}) and (\ref{d rou}) are still well defined. The IST
Hankel operator $\mathbb{H}\left(  x,t\right)  $ however need no longer be in
the Sarason algebra\footnote{Even boundedness of $\mathbb{H}\left(
x,t\right)  $ may in fact be lost.} resulting in a lack of compactness. On the
other hand the symbol of $\mathbb{H}\left(  x,t\right)  $ clearly says what
the problems are and what can be tried to approach them. For instance, if
$q\left(  x\right)  =O\left(  1/x^{2}\right)  ,\ \left\vert x\right\vert
\rightarrow\infty$, then $\mathbb{L}_{\chi_{a}q}$ may have infinite many
negative bound states for any $a$ but the right reflection coefficient
$R_{a}\left(  k\right)  $ off $q_{a}$ is well behaved for every $k\neq0$ and
at $k=0$ its argument may have a jump discontinuity of size $\gamma$ readily
available from the asymptotic behavior of $q$. The infinite negative spectrum
can be handled by applying an infinite chain of Darboux transforms
\cite{DegSha94}. The jump discontinuity at zero can be modeled by the suitably
defined analytic function $\left(  \dfrac{k-i\varepsilon}{k+i\varepsilon
}\right)  ^{\gamma/2\pi}$,$\ $with any $\varepsilon>0$, allowing us to "factor
out" the undesirable behavior. This way (\ref{det_form}) could be effectively
regularized by singling out the behavior of $q\left(  x,t\right)  $
corresponding to the point $k=0$. Similarly, one can approach Wigner-von
Neumann\ initial profiles $q\left(  x\right)  \sim A\ \dfrac{\sin2\omega x}%
{x},\ \left\vert x\right\vert \rightarrow\infty$. A new additional feature
emerges \cite{Klaus91} here. It is related to the so-called Wigner-von Neumann
resonance\footnote{Which could under certain condition be an embedded positive
bound state.} $\omega^{2}$ and says that the $\arg R\left(  k\right)  $ has a
jump discontinuity at $\pm\omega$ of size $\gamma=2\pi\omega\left\vert
A\right\vert $. As in the previous example, the singular behavior of $R\left(
k\right)  $ is captured by suitably defined function $\left(  b_{\mu}\left(
k\right)  b_{-\overline{\mu}}\left(  k\right)  \right)  ^{\gamma/2\pi}$, where
$b_{\mu}$ is the Blaschke factor with zero $\mu=\omega+i\varepsilon$ for any
$\varepsilon>0$. The formula (\ref{det_form}) gains an extra term associated
with the Wigner-von Neumann resonance. Same way a sum of different Wigner-von
Neumann potentials can be handled. We have not worked out the details~even for
one resonance initial profile. But this would be particularly interesting as
the Matveev's conjecture \cite{Matveev2002} on existence of bounded positon
(breather)\ solutions could be then addressed. One of the challenges is that
the spectral and scattering theory for long range potentials is not as
well-developed as its short-range counterpart and there is no 'one stop
shopping' like the seminal \cite{Deift79}.

\subsection{Asymptotic solutions}

Much of the activity related to the Cauchy problem (\ref{KdV})-(\ref{KdVID})
is concerned with long time asymptotic behavior of its solution. The most
powerful method to study this is arguably the nonlinear steepest descent
method \cite{DeiftZhou1993} based on the Riemann-Hilbert problem. A nice
well-written exposition of this method for (\ref{KdV})-(\ref{KdVID}) is given
in recent \cite{GrunertTeschl09}. Roughly speaking, this approach amounts to
taking one (say right) basic scattering relation (\ref{basic scatt identity})
and its complex conjugate and considering this pair as a two by two matrix
Riemann-Hilbert problem for the row matrix $Y=\left(
\begin{array}
[c]{ccc}%
Ty_{-} &  & y_{+}%
\end{array}
\right)  $. Asymptotics in each region of $\left(  x,t\right)  $ is then
extracted from a very clever multi-step transformation (factorization,
conjugation, contour deformation, etc.) of the original Riemann-Hilbert
problem to the one that captures the asymptotic behavior in that particular
region. No Hankel or Toeplitz operator explicitly appear this way but it is of
course well-known \cite{ClanceyGohbergBOOK} that the Riemann-Hilbert problem
is essentially equivalent to invertibility of a certain Toeplitz
operator\footnote{As a matter of fact, Theorem \ref{thGru01}, of principal
importance to us, was originally found due to some problems having roots in
the Riemann-Hilbert problem \cite{Govorov94}, \cite{Ostrovski92}.}. It is
therefore reasonable to ask if asymptotic solutions of (\ref{KdV}%
)-(\ref{KdVID}) could be obtained entirely within the theory of
Hankel/Toeplitz operators as effectively as using techniques of the
Riemann-Hilbert problem? We don't have a clear vision if this could be the
case. However translating the nonlinear steepest descent method into the
language of Hankel would be of interest in its own right. We hope that more
and deeper connections between soliton theory and the theory of Hankel
operators could be uncovered this way.

It is of course one of our goals to analyze asymptotics of (\ref{KdV}%
)-(\ref{KdVID}) with initial data subject to Hypothesis \ref{hyp1.1}. A
comprehensive treatment of the case $q\left(  x\right)  \rightarrow-h^{2}$ as
$x\rightarrow-\infty$ in all asymptotic regions has been recently given in
\cite{EgorovaetalNon2013}. We also refer to \cite{EgorovaetalNon2013} for a
extensive literature review. The main feature of this case is that such a step
asymptotically splits into an infinite train of solitons twice as high as the
step itself\footnote{In nature this phenomenon can be seen in the so-called
undular bore waves.}. We cautiously conjecture that this type of behavior is
universal as long as $\mathbb{L}_{q}$ has some negative essential spectrum.
This also holds \cite{KK94} if $q\left(  x\right)  $ tends to a periodic
function at $-\infty$. It can be quite easily seen from the considerations of
Subsection \ref{rate of decay} that $q\left(  x,t\right)  \rightarrow0$
uniformly in the region $x/t>h_{0}^{2}$ as $t\rightarrow\infty$. This is as
much as we know at this point. We only mention that both, right and left,
basic scattering relations are used in \cite{EgorovaetalNon2013} to implement
the nonlinear steepest descent method. In our case the left scattering
relation need not exist.

\subsection{Analyticity}

Theorem \ref{MainThm} says that if $0<\delta<2$ then $q\left(  x,t\right)  $
is meromorphic on the whole complex plane for any $t>0$ and hence, as any
meromorphic on $\mathbb{C}$ function, it is completely characterized by a
countable number of time dependent parameters. Viewing a pure soliton solution
as a meromorphic function of $x$ has a long story. In particular, the
importance of pole dynamics was recognized by Kruskal\ \cite{Kruskal74} back
in the early 70s for pure soliton solutions and has been actively studied
since then (see also \cite{Air77}, \cite{Bona2009}, \cite{Segur2000},
\cite{GWeikard06} to mention just four). But very little is known about
meromorphic solutions to the Cauchy problem (\ref{KdV})-(\ref{KdVID}). The
problem boils down to the study of the meromorphic operator valued function
$\left(  \mathbb{I}+\mathbb{H}\left(  x,t\right)  \right)  ^{-1}.$ The general
theory \cite{Steinberg1969} only says that its poles depend continuously on
$t$ and cannot appear or disappear. We are unaware of any relevant helpful
results in the theory of Hankel operators.

\subsection{Unification of the short-range and periodic ISTs}

This is posed in \cite{AC91} as an intriguing open problem. Why the periodic
IST implies the short-range IST is explained in the classical book
\cite{NovikovetalBook} (see also \cite{Venak88} and \cite{ErcMcKean90} for
rigorous related results) but how the periodic IST emerges from the short
range one does not appear to be well understood. We believe that to approach
this problem one needs to consider (\ref{det_form}) for a periodic function
restricted to $\left(  -\infty,a\right)  $ and then let $a\rightarrow\infty$.
We see two reasons why this approach should work: (1) the reflection
coefficient off $\left(  -\infty,a\right)  $ is explicitly representable
\cite{PavSmirn82} in terms of spectral characteristics of the periodic
problem; (2) the solution formula for the periodic case \cite{NovikovetalBook}
has a similar to (\ref{det_form}) structure with the tau function represented
via the theta function. This suggests that the convergence should indeed take
place. However we do not have any insight how difficult this problem could be.
We cautiously suspect that the ideas and machinery used to find asymptotics of
the so-called Toeplitz determinants \cite{BotSil06} could be applied to our
problem. If this indeed can be done within the theory of Hankel/Toeplitz
operators then we might be able to avoid the complicated Riemann surfaces
techniques commonly associated with the periodic IST.

\end{document}